\pdfoutput=1
\documentclass[]{amsart}
\usepackage{amsthm}
\usepackage{amsfonts}
\usepackage{amssymb}
\usepackage{amsmath}
\usepackage{verbatim}
\usepackage{mathtools}
\usepackage{setspace}
\usepackage{array}   
\usepackage{url}
\newcolumntype{L}{>{$}l<{$}} 
\newcolumntype{C}{>{$}c<{$}} 
\newcolumntype{R}{>{$}r<{$}} 
\newtheorem{thm}{Theorem}[section]
\newtheorem{prop}[thm]{Proposition}
\newtheorem{lemma}[thm]{Lemma}
\newtheorem{cor}[thm]{Corollary}

\theoremstyle{definition}
\newtheorem{defin}[thm]{Definition}
\newtheorem{examp}[thm]{Example}
\theoremstyle{remark}
\newtheorem{rmk}[thm]{Remark}
\newcommand{\normtext}[1]{\text{\normalfont{#1}}}
\newcommand{\spec}[1]{\text{Spec}(#1)}
\newcommand{\equivBun}{\normtext{Bun}_Y^{\Gamma, G}(\boldsymbol{\tau})}
\newcommand{\equivBunSS}{\normtext{Bun}_Y^{\Gamma, G}(\boldsymbol{\tau}, ss)}

\newcommand{\equivBunPTwo}[1]{\normtext{Bun}_Y^{\Gamma, G;P}(\boldsymbol{\tau};\boldsymbol{\tau}_u, #1)}
\newcommand{\equivBunPTwoSS}[1]{\normtext{Bun}_Y^{\Gamma, G;P}(\boldsymbol{\tau}, ss;\boldsymbol{\tau}_u, #1)}
\newcommand{\equivBunL}[1]{\normtext{Bun}_Y^{\Gamma, L}(\boldsymbol{\tau}_u, #1)}
\newcommand{\equivBunLP}{\normtext{Bun}_Y^{\Gamma, L'}(\boldsymbol{\tau}_u')}

\title{Rank Reduction of Conformal Blocks}
\author{Michael Schuster}
\thanks{University of North Carolina, Chapel Hill}

\begin{document}
\begin{abstract}
Let $X$ be a smooth, pointed Riemann surface of genus zero, and $G$ a simple, simply-connected complex algebraic group.
Associated to a finite number of weights of $G$ and a level is a vector space called the space of conformal
blocks, and a vector bundle of conformal blocks over $\overline{\text{M}}_{0,n}$.  We show that, assuming the weights are on a
face of the multiplicative eigenvalue polytope, the space of conformal blocks is isomorphic to a product of conformal
blocks over groups of lower rank.  If the weights are on a degree zero wall, then we also show that there is an isomorphism
of conformal blocks bundles, giving an explicit relation between the associated nef divisors.  The methods of the proof
are geometric, and use the identification of conformal blocks with spaces of generalized theta functions, and the moduli 
stacks of parahoric bundles recently studied by Balaji and Seshadri.  
\end{abstract}
\maketitle
\section{Introduction}
\label{IntroSec}

Conformal blocks are vector spaces $\mathcal{V}^{\dagger}_{\mathfrak{g}, \vec{\lambda}, \ell}(X, \vec{p})$ 
associated to a stable pointed curve $(X, \vec{p})$, a simple Lie algebra $\mathfrak{g}$, a finite set of dominant integral weights $\lambda_1, \ldots, \lambda_n$, 
and a positive integer $\ell$ called the level.  Originally arising from physics in conformal field theory, conformal
blocks were later shown to be isomorphic to spaces of generalized theta functions \cite{SNR,LASOR}.  The dimension of 
conformal blocks does not depend on the pointed curve, and they form 
vector bundles $\mathbb{V}_{\mathfrak{g}, \vec{\lambda}, \ell}$ over
families of curves.  In genus zero these vector bundles correspond to base point free divisors on the
moduli space $\overline{\text{M}}_{0,n}$ of genus zero stable pointed curves.  Fakhruddin recently proved formulas for
conformal blocks divisors in terms of boundary divisors, generating interest in using conformal blocks divisors
to study the geometry of $\overline{\text{M}}_{0,n}$ \cite{FAK}.  Conformal blocks divisors have been shown
to give rise to interesting contraction maps \cite{GG, GJDM}, and there are a number of
open problems related to these divisors \cite{FAK,BGM15,BGM16}.

If $X \cong \mathbb{P}^1$, and assuming $\mathcal{V}^{\dagger}_{\mathfrak{g}, \vec{\lambda}, \ell}(X, \vec{p})$ is non-zero, 
we can associate to $(\vec{\lambda}, \ell)$
a point $\vec{\mu}$ in the \emph{multiplicative eigenvalue polytope} $\Delta_n(G)$, where $G$ is the simply connected algebraic 
group associated to $\mathfrak{g}$.  The multiplicative polytope is the set of solutions of the multiplicative eigenvalue
problem, a version of the additive eigenvalue problem first solved by Klyachko for $G=\text{SL}_{r+1}$ \cite{KLY98}. It has attracted interest because of its 
connection to a number of important problems in linear algebra, representation theory,
the quantum cohomology of flag varieties, moduli spaces of parabolic
bundles, and holonomy of flat vector bundles \cite{AW98,BEL01,TWPB,BK}.  The goal of this paper is to study conformal
blocks when $\vec{\mu}$ is on the boundary of the multiplicative polytope.

Our main result is that conformal blocks \emph{factor} when $\vec{\mu}$ is on a regular facet of the multiplicative
polytope, with the factors arising from groups of reduced rank.
\begin{thm}
If $\mathcal{V}^\dagger_{\mathfrak{g}, \vec{\lambda}, \ell} = \mathcal{V}^\dagger_{\mathfrak{g}, \vec{\lambda}, \ell}(\mathbb{P}^1, \vec{p})$ is non-zero, and the associated $\vec{\mu}$ is on a regular facet of 
the multiplicative polytope, then the space of conformal blocks factors:
\[
\mathcal{V}^{\dagger}_{\mathfrak{g}, \vec{w}} \cong \mathcal{V}^{\dagger}_{\mathfrak{g}_1, \vec{w}_1} \otimes \mathcal{V}^{\dagger}_{\mathfrak{g}_2, \vec{w}_2} 
\]
where $\mathfrak{g}_1 \times \mathfrak{g}_2 \subseteq \mathfrak{g}$ is  subalgebra of semisimple rank one less than $\mathfrak{g}$.
\end{thm}
On certain facets we also get an isomorphism of conformal blocks bundles.
\begin{cor}
If $\vec{\mu}$ is on a degree zero facet, then the above isomorphism extends to a vector bundle isomorphism
over $\overline{\normtext{M}}_{0,n}$, giving a relation of conformal blocks divisors:
\[
\mathbb{D}_{\mathfrak{g}, \vec{w}} = \text{rk}(\mathbb{V}_{\mathfrak{g}_2, \vec{w}_2}) \cdot \mathbb{D}_{\mathfrak{g}_1, \vec{w}_1} + \text{rk}(\mathbb{V}_{\mathfrak{g}_1, \vec{w}_1}) \cdot \mathbb{D}_{\mathfrak{g}_2, \vec{w}_2}
\]
\end{cor}

These results can be seen as a generalization of factorization results for Littlewood-Richardson coefficients, proven for $\mathfrak{sl}_{r+1}$
by King, Tollu, and Toumazet \cite{KRT}, then in all types by Roth \cite{RTH}; Ressayre has generalized Roth's result to general branching coefficients \cite{RESS12}.
The assumption on the weights is that they lie on a face of the \emph{additive eigencone}, which is an analogue of the 
multiplicative polytope for Lie algebras. 

To prove the reduction theorem, we use the fact that spaces of conformal blocks can be canonically
identified with spaces of generalized theta functions.  More precisely, there is a line bundle $\mathcal{L}_{\vec{w}}$ over
the moduli stack of parabolic bundles $\text{Parbun}_G$ such that 
$\text{H}^0(\text{Parbun}_G, \mathcal{L}_{\vec{w}}) \cong \mathcal{V}^{\dagger}_{\mathfrak{g}, \vec{w}}$.
Parabolic bundles are principal $G$-bundles together with extra data over each point $p_i$, see section \ref{CBDef}
for a precise definition.  For arbitrary weight data, we need to work over a generalization of the stack
of parabolic bundles: the moduli stack of parahoric bundles.

Parahoric bundles are torsors over a smooth group scheme $\mathcal{G} \rightarrow X$ determined by the choice of weight 
data $\vec{w}$.  Moduli stacks of parahoric bundles are the natural setting in which work with conformal
blocks as generalized theta functions when the weight data is on the affine wall of the alcove $\mathcal{A}$.
When $G = \text{SL}_{r+1}$, parahoric bundles can be identified with parabolic vector bundles where the
underlying vector bundle has nonzero degree.  For a general group, parahoric bundles can be more exotic.

Nevertheless, using the identification of parahoric bundles with equivariant bundles proven by Balaji and Seshadri \cite{BAS},
we have the following theorem.
\begin{thm}\label{GlblSecDesc}
Let $\mathcal{G}$ be the parahoric group scheme associated to the weight data $\vec{w}$ over a smooth projective
curve of arbitrary genus $X$. Then the line bundle 
$\mathcal{L}_{\vec{w}}$ descends to $\normtext{Bun}_{\mathcal{G}}$, and 
$\normtext{H}^0(\normtext{Bun}_{\mathcal{G}}, \mathcal{L}_{\vec{w}})$ is naturally isomorphic
to the space of conformal blocks $\mathcal{V}^{\dagger}_{\mathfrak{g}, \vec{w}}(X,\vec{p})$.
\end{thm}
\begin{rmk}
The parahoric group schemes in the above theorem are those associated
to the split group $G(K)$, where $K=k((z))$ is the field of formal Laurent series.
\end{rmk}

In the remainder of the introduction, we discuss in more detail the geometry of the multiplicative polytope,
give precise statements of the reduction theorem, and outline the proof of these results.

\subsection{Reduction rules on the regular faces}

Let $G$ be a simple, simply connected algebraic group over $\mathbb{C}$ of rank $r$.  Choose a maximal torus $T$ and Borel subgroup $B$ of $G$.
Let $K \subseteq G$ be a maximal compact subgroup.  
Let $\mathfrak{g}$ be the Lie algebra of $G$, and $\mathfrak{h} \subseteq \mathfrak{g}$ the Cartan subalgebra associated to $T$.  Let $\alpha_1, \ldots, \alpha_r \in \mathfrak{h}^*$
be the simple roots of $\mathfrak{g}$, and let $\theta$ be the highest root.	
Then the \emph{fundamental alcove} $\mathcal{A}$ is defined as follows:
\[
\mathcal{A} = \{\mu \in \mathfrak{h} | \theta(\mu) \leq 1, \alpha_i(\mu) \geq 0 \text{ for all } i  \}.
\]
The \emph{multiplicative polytope} is defined as
\[
\Delta_n = \{\vec{\mu} \in  \mathcal{A}^n | \text{Id} \in C(\mu_1) \cdots C(\mu_n) \},
\]
where $C(\mu)$ denotes the conjugacy class of $\text{exp}(2\pi i \mu)$ in $K$.  This subset was shown to 
be a convex polytope in \cite{MW98} using symplectic methods.

Now to weights $\lambda_1, \ldots, \lambda_n$ and a level $\ell$ we can associate points of the fundamental alcove
$\mu_i = \frac{\kappa(\lambda_i)}{\ell} \in \mathcal{A}$, where $\kappa: \mathfrak{h}^* \xrightarrow{\sim} \mathfrak{h}$ is the
isomorphism induced by the (normalized) Killing form.  Fix distinct points $p_1, \ldots, p_n$ in $X = \mathbb{P}^1$.
Then the space $\mathcal{V}^{\dagger}_{\mathfrak{g}, \vec{\lambda}, \ell} = \mathcal{V}^{\dagger}_{\mathfrak{g}, \vec{\lambda}, \ell}(X, \vec{p})$
has positive dimension (possibly after scaling the weights and level) if and only if the tuple 
$(\mu_1, \ldots, \mu_n)$ lies in the multiplicative polytope.  
We say that the weight data $\vec{w} = (\lambda_1, \ldots, \lambda_n, \ell)$ 
is on a face of the multiplicative polytope if the associated $(\mu_1, \ldots, \mu_n)$ is on the face.

The effect of weights being on a face of the multiplicative polytope is a reduction of the problem
to a lower rank group.  For example, if $\vec{\mu} \in \Delta_n$ is
on a regular facet, then one can find $A_i \in C(\mu_i)$ such that $A_1 \cdots A_n = \text{Id}$ and each $A_i$ is
block diagonal of the same dimensions.  See Knutson's proof \cite{KNT00} of a similar result for Hermitian
matrices.

\subsubsection{Regular faces of the multiplicative polytope}

As a subset of $\mathcal{A}^n$, the multiplicative polytope is determined by its \emph{regular facets},
that is, the facets intersecting the interior of $\mathcal{A}^n$.
The problem of finding the inequalities defining the faces of $\Delta_n$ has a long history.  The general
form of the answer is that the inequalities are parametrized by certain cohomology products in the cohomology
ring of Grassmannians $G/P$.  It was first solved for $\text{SL}_2$ by Biswas \cite{BIS98}, then by Agnihotri and Woodward
for $\text{SL}_n$ \cite{AW98} and independently by Belkale in \cite{BEL01}; Belkale furthermore reduced the inequalities to an 
irredundant set.  
Teleman and Woodward \cite{TWPB} found inequalities defining $\Delta_n$ in general type, and more recently Belkale 
and Kumar \cite{BK} reduced these inequalities to an irredundant set, building on their work on the additive eigencone 
\cite{BK06} and Ressayre's proof of the irredundancy of Belkale and Kumar's inequalities in \cite{RSS10}.

Teleman and Woodward showed that $\Delta_n$ is determined by
a set of inequalities parametrized by (small) quantum cohomology products in $\text{QH}^*(G/P)$ of the form 
$\sigma_{u_1} \ast \cdots \ast \sigma_{u_n} = q^d [pt]$ for all maximal parabolics $P$.  We call the set of 
faces of the multiplicative polytope associated to these products \emph{TW-faces}, which include the
non-empty regular faces, but in general could include faces that do not intersect the interior of $\mathcal{A}^n$.
Our main theorem assumes that a tuple $\vec{\mu} \in \Delta_n$ associated
to the space of conformal blocks lies on a TW-face associated to one of these products.
For a precise definition of the quantum product and further discussion of its relationship with the multiplicative 
polytope, see section \ref{MULTQUANT}.

\subsubsection{Main theorem for degree zero walls}
Now assume $\vec{w}$ is on the face of the multiplicative polytope associated to the cohomology product
$\sigma_{u_1} \ast \cdots \ast \sigma_{u_n} = \lbrack \text{pt} \rbrack \in \normtext{QH}^*(G/P)$.
Note that since $d=0$, this corresponds to a product in $\text{H}^*(G/P)$.
Let $L \subseteq P$ be the Levi factor containing the chosen maximal torus $T$ of $G$, and let $L' = [L, L]$.  
Then $L'$ is semisimple and simply connected, and therefore is isomorphic to a product of simple groups; for simplicity
we assume that there are two factors: $L' \cong G_1 \times G_2$.  Then our main theorem gives an isomorphism between
$\mathcal{V}^{\dagger}_{\mathfrak{g}, \vec{\lambda}, \ell}$ and conformal blocks associated to $G_1$ and $G_2$, which
together have one less rank than $G$.

\begin{thm}\label{MNTHM}
For weight data $\vec{w} = (\lambda_1, \ldots, \lambda_n, \ell)$ in the multiplicative polytope, lying on the face corresponding to
$\sigma_{u_1}  \cdots  \sigma_{u_n} = \lbrack \text{pt} \rbrack \in \normtext{H}^*(G/P)$, we have a 
natural isomorphism of conformal blocks
\[
\mathcal{V}^{\dagger}_{\mathfrak{g}, \vec{w}} \cong \mathcal{V}^{\dagger}_{\mathfrak{g}_1, \vec{w}_1} \otimes \mathcal{V}^{\dagger}_{\mathfrak{g}_2, \vec{w}_2} 
\]
where $\vec{w}_1$ and $\vec{w}_2$ are the restrictions of following weight data to $\mathfrak{g}_1$ and $\mathfrak{g}_2$:
\begin{enumerate} 
\item Weights $u_1^{-1} \lambda_1, \ldots, u_n^{-1} \lambda_n$.
\item Levels $m_1 \ell$ and $m_2 \ell$, where $m_1$ and $m_2$ are the Dynkin indices of $\mathfrak{g}_1$ and $\mathfrak{g}_2$ 
in $\mathfrak{g}$, respectively.
\end{enumerate}
\end{thm}

\begin{rmk}
The Dynkin indices $m_1$ and $m_2$ for simply-laced groups are always equal to $1$.
\end{rmk}

A simple argument shows that we can extend this isomorphism to conformal blocks \emph{bundles}, 
which are vector bundles $\mathbb{V}_{\vec{w}}$ over the moduli space $\overline{\text{M}}_{0,n}$ of genus zero
stable pointed curves with $n$ marked points,  
such that the fiber over $(X, p_1, \ldots, p_n) \in \overline{\text{M}}_{0,n}$ is the dual
of the space of conformal blocks, denoted $\mathcal{V}_{\mathfrak{g}, \vec{w}}$.  The proof uses the 
fact that in genus $0$ these bundles are globally generated, and Roth's reduction theorem for
invariants \cite{FAK,RTH}.

\begin{cor}\label{CBBCOR}
With the same assumptions as above, we have an isomorphism of conformal blocks bundles on $\overline{\normtext{M}}_{0,n}$:
\[
\mathbb{V}_{\mathfrak{g}, \vec{w}} \cong \mathbb{V}_{\mathfrak{g}_1, \vec{w}_1} \otimes \mathbb{V}_{\mathfrak{g}_2, \vec{w}_2}.
\]
Therefore the divisors $\mathbb{D}_{\mathfrak{g}, \vec{w}}$ given by the first Chern classes of these vector bundles satisfy the relation
\[
\mathbb{D}_{\mathfrak{g}, \vec{w}} = \text{rk}(\mathbb{V}_{\mathfrak{g}_2, \vec{w}_2}) \cdot \mathbb{D}_{\mathfrak{g}_1, \vec{w}_1} + \text{rk}(\mathbb{V}_{\mathfrak{g}_1, \vec{w}_1}) \cdot \mathbb{D}_{\mathfrak{g}_2, \vec{w}_2}
\]
\end{cor}

Finally, since the inequality associated to $\sigma_{u_1}  \cdots \sigma_{u_n} =\lbrack \text{pt} \rbrack$ 
does not depend on the level $\ell$, we can increase $\ell$ and the weight data will still be on the corresponding face
of the multiplicative polytope.  In fact these faces also define a cone called the \emph{additive eigencone} (see section \ref{MULTQUANT}).
  It is well known that at high enough level conformal blocks become isomorphic to spaces of invariants 
$\mathbb{A}_{\mathfrak{g}, \vec{\lambda}} \vcentcolon= (V_{\lambda_1} \otimes \cdots \otimes V_{\lambda_n})^{\mathfrak{g}}$.
Therefore as a final corollary we get Roth's reduction theorem for tensor product invariants:

\begin{cor}(\cite{RTH})
Given weights $\vec{\lambda}$ lying on a face of the additive eigencone corresponding to 
$\sigma_{u_1}  \cdots  \sigma_{u_n} =\lbrack \text{pt} \rbrack$, we have a canonical
isomorphism of invariants:
\[
\mathbb{A}_{\mathfrak{g}, \vec{\lambda}} \cong \mathbb{A}_{\mathfrak{g}_1, \vec{\lambda}_1} \otimes \mathbb{A}_{\mathfrak{g}_2, \vec{\lambda}_2},
\]
where $\vec{\lambda}_1$, $\vec{\lambda}_2$ are given by the restrictions of $u_i^{-1} \lambda_i$ to $\mathfrak{g}_1$ and $\mathfrak{g}_2$.
\end{cor}

\subsubsection{Reductions on positive degree faces}

We also obtain reductions on the positive degree TW-faces of the multiplicative polytope.  
We continue with the notation from the previous section.  To the Levi subgroup $L \subseteq P$
we assign a degree $k_L$, which is the size of the kernel of the isogeny $Z_0 \rightarrow L/L'$, where $Z_0$ is the connected
component of the identity of $L$.  In \cite{BK}, Belkale and Kumar showed the existence of a cocharacter $\mu_P$
lying in the fundamental alcove of $L$, such that $|\omega_P(\mu_P)| = 1$, where $\omega_P$ is the fundamental
weight associated to $P$.  Let $d_0$ be the smallest integer such that $d+d_0\omega_P(\mu_P) \equiv 0 \text{ (mod } k_L)$.
Then by adding $d_0$ points to our pointed curve and twisting using the cocharacter $\mu_P$, we obtain 
a rank reduction theorem for weight data on positive degree facets (see Proposition \ref{L'Red} for more details).

\begin{thm}\label{MNTHM2}
For weight data $\vec{w} = (\lambda_1, \ldots, \lambda_n, \ell)$ in the multiplicative polytope, lying on the facet corresponding to
$\sigma_{u_1} \ast \cdots \ast \sigma_{u_n} =q^d \lbrack \text{pt} \rbrack \in \normtext{QH}^*(G/P)$, we have a 
natural isomorphism of conformal blocks
\[
\mathcal{V}^{\dagger}_{\mathfrak{g}, \vec{w}} \cong \mathcal{V}^{\dagger}_{\mathfrak{g}_1, \vec{w}_1} \otimes \mathcal{V}^{\dagger}_{\mathfrak{g}_2, \vec{w}_2} 
\]
where $\vec{w}_1$ and $\vec{w}_2$ are the restrictions of following weight data to $\mathfrak{g}_1$ and $\mathfrak{g}_2$:
\begin{enumerate} 
\item The first $n$ weights are $u_1^{-1} \lambda_1, \ldots, u_n^{-1} \lambda_n$.
\item The last $d_0$ weights are $\ell\mu_P^*$, where $\mu_P^*$ is the dual with respect to the Killing form.
\item The levels are $m_1 \ell$ and $m_2 \ell$, where $m_1$ and $m_2$ are the Dynkin indices of $\mathfrak{g}_1 \times \mathfrak{g}_2$ 
in $\mathfrak{g}$.
\end{enumerate}
\end{thm}
\begin{rmk}
This result of course has no classical analogue for spaces of invariants. Furthermore
Example \ref{POSDEGEX} shows that this isomorphism does not extend to conformal blocks vector bundles in general.  It would
be interesting to know if there is a relationship between these bundles.  
\end{rmk}

\subsection{Outline of proof of the main theorem}

Now we outline the proof of the reduction theorem.  
The proof is simplest when the weight data $\vec{w}$ lies in the interior of the alcove, and the degree of
the wall is zero.  An outline in the general case is given at the beginning of section
\ref{PReductionSec}, after parahoric bundles have been introduced.
Assume that the weights lie on a degree zero facet of the multiplicative polytope corresponding to 
a cohomology product $\sigma_{u_1} \cdots \sigma_{u_n} = [pt] \in \text{H}^*(G/P)$.  

Now let $\text{Parbun}_G$ be the moduli stack of parabolic bundles with full flags over $(X, p_1, \ldots, p_n)$. 
To our weight data $\vec{w}$ we can associate a line bundle $\mathcal{L}_{\vec{w}}$ over
$\text{Parbun}_G$.  Then the space of conformal blocks $\mathcal{V}^{\dagger}_{\mathfrak{g}, \vec{w}}$ can be identified
with the space of global sections of $\mathcal{L}_{\vec{w}}$.  The first step is to show that the following morphism of stacks 
induces an isomorphism of spaces of global sections of $\mathcal{L}_{\vec{w}}$ via pullback by
\[
\text{Parbun}_G \xleftarrow{\iota} \text{Parbun}_L(0),
\]
where the parabolic $L$-bundles are degree $0$, and $\iota$ is given by extension of structure group and by twisting 
the flags over each $p_i$ by $u_i$.  The twisting by the $u_i$'s makes the pullback $\iota^*\mathcal{L}_{\vec{w}}$
isomorphic to $\mathcal{L}_{\vec{w}'}$ where $\vec{w}'$ is the weight data described in the reduction theorem.

Therefore we want to show that global sections of $\mathcal{L}_{\vec{w}}$ over $\text{Parbun}_L(0)$ extend to global sections
of $\text{Parbun}_G$ uniquely.  To show this we use a method originally due to Ressayre \cite{RSS10}: we
use another stack $\mathcal{C}$, with morphisms $\pi: \mathcal{C} \rightarrow \text{Parbun}_G$, and
$\xi: \mathcal{C} \rightarrow \text{Parbun}_L(0)$.  
\[
\includegraphics{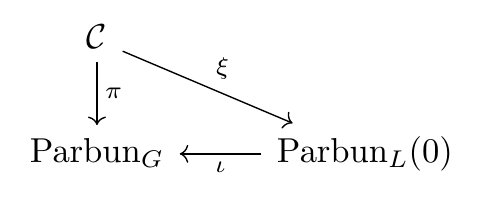}
\]
The fibers of $\mathcal{C}$ over $\widetilde{\mathcal{E}} \in \text{Parbun}_G$ are the  degree $d$ $P$-reductions
of $\mathcal{E}$ with \emph{relative position} $u_1, \ldots, u_n$ (relative to each flag).  Since
$\sigma_{u_1}  \cdots  \sigma_{u_n} = [pt]$, generically $\pi$ is one-to-one,
and in fact is birational (see \cite{BK}).  The morphism $\xi$ in terms of vector bundles $W \subset V$ and flags $F_i^{\bullet}$ in type A is
given by $(V, W, F_i^\bullet) \mapsto W \oplus V/W$  together natural induced flags on $W \oplus V/W$;
this morphism is surjective.  While the above diagram is not $2$-commutative, the pullbacks of $\mathcal{L}_{\vec{w}}$ via $\pi$ and 
$\iota \circ \xi$ can be identified over $\mathcal{C}$.  This identification depends on the weight data $\vec{w}$
being on the facet corresponding to $\sigma_{u_1}  \cdots  \sigma_{u_n} = [pt]$.

The next step of the proof is to show that $\pi$ is proper over the semistable locus of $\text{Parbun}_G$
with respect to $\mathcal{L}_{\vec{w}}$.  Starting with a one parameter family of semistable parabolic bundles
$\mathcal{E} \rightarrow X \times C$ and a family of $P$-reductions of parabolic degree $0$ over the punctured curve $C^*$,
we embed and complete the family of $P$-reductions inside a Hilbert scheme.  Then properness follows
from a no-ghosts theorem proved by Holla and Narasimhan in \cite{HNA}.
Finally, we use Zariski's main theorem to show the pullback of global sections of $\mathcal{L}$ via $\pi$ induces an isomorphism,
and finish the reduction $\text{H}^0(\text{Parbun}_G, \mathcal{L}_{\vec{w}}) \xrightarrow{\sim} \text{H}^0(\text{Parbun}_L(0), \mathcal{L}_{\vec{w}'})$ 
with a simple diagram chase.  For more details about the stack $\mathcal{C}$ and the proof
of the properness of $\pi$ over the semistable locus, see section \ref{PReductionSec}.  For the details
about $\iota$ and $\xi$, and the reduction of conformal blocks
to $\text{Parbun}_L(0)$, see the beginning of section \ref{ProofSec}.

Finally, we need to reduce to the derived subgroup $L'=[L,L]$ to finish the proof of the reduction theorem.
Again, we use a morphism of stacks
\[
\text{Parbun}_L(0) \xleftarrow{\iota'} \text{Parbun}_{L'}
\]
where $\iota'$ is given by extension of structure group.  Then by a straightforward argument in section 7
of \cite{BK}, since $\vec{w}$ is on the facet corresponding to $\sigma_{u_1}  \cdots \sigma_{u_n} = [pt]$,
$\iota'$ induces an isomorphism $\text{H}^0(\text{Parbun}_L(0), \mathcal{L}_{\vec{w}'}) \xrightarrow{\sim} \text{H}^0(\text{Parbun}_{L'}, \mathcal{L}_{\vec{w}'})$,
finishing the proof of the reduction theorem, since $\text{H}^0(\text{Parbun}_{L'}, \mathcal{L}_{\vec{w}'})$ can be identified with
a product of conformal blocks.  This step requires more care when $d>0$; for details see the discussion in section \ref{ProofSec}.

This strategy requires some modifications when the weights are not in the interior of the fundamental alcove.
In this case $\mathcal{C}$ is not proper over $\text{Parbun}_G$, and we need to enlarge $\mathcal{C}$
to a larger stack $\mathcal{Y}$ of $P$-reductions.  Unfortunately there is no extension of $\xi$ to $\mathcal{Y}$.
There is however a morphism $\xi: \mathcal{Y} \rightarrow \text{Bun}_\mathcal{G}$, where $\text{Bun}_\mathcal{G}$
is a moduli stack of \emph{parahoric bundles}.  For this reason, to prove the reduction theorem for
arbitrary weights, we need to work over stacks of parahoric bundles rather than parabolic bundles. 
Having made this change, the proof follows in essentially the same way as above.

\subsection{Outline of the paper}

The paper is organized as follows.  In section \ref{PolySec}, we review some background material on the multiplicative
polytope, quantum cohomology, conformal blocks, and parabolic bundles.  

In section \ref{ParahoricSec} we review the work of Balaji and Seshadri on parahoric
bundles, semistability of parahoric bundles, and their moduli spaces, and prove Theorem \ref{GlblSecDesc}.  
To show this we study the closed fiber of the Bruhat-Tits group schemes, and identify the fibers of the
morphism $\text{Parbun}_G \rightarrow \text{Bun}_{\mathcal{G}}$ as certain flag varieties determined by
the weight data.

In section \ref{PReductionSec} we start with an outline of the proof of the reduction theorem (Theorem \ref{MNTHM2})  in terms of equivariant
bundles.  We then construct $\mathcal{Y}$, and show that $\pi: \mathcal{Y} \rightarrow \text{Bun}_\mathcal{G}$ is proper 
over the semistable locus.  Key to the properness proof is the lifting of families $P$-reductions
of parabolic bundles to equivariant $G$-bundles over a ramified cover $Y \rightarrow X$.  This process in general
is discontinuous, but it is sufficient to assume that the family of $P$-reductions is of constant relative position.

Finally, in section \ref{ProofSec} we use the results of the previous chapters to prove the reduction result as outlined above.  We 
also prove that when $d=0$, we get an isomorphism of conformal blocks bundles.

\subsection{Notation and conventions}  

Let $X$ be a smooth, connected and projective curve over $k = \mathbb{C}$ of genus $g$, and $p_1, \ldots, p_n$
be distinct points in $X$.  Starting in section \ref{PReductionSec}, we will assume $X \cong \mathbb{P}^1$.
Assume we have sufficiently many points to make the automorphism
group of this pointed curve finite.  All schemes and algebraic stacks are defined over $k$.  
For a morphism of algebraic stacks $f: \mathcal{X} \rightarrow \mathcal{Y}$, we say $f$ is representable
if it is representable by schemes.

Let $G$ be a simply connected simple algebraic group over $k$.  Fix a Borel 
subgroup $B$, and a maximal torus $T$.  Let $W=N(T)/T$
be the Weyl group.  For a standard parabolic $P \supseteq B$ let $U = U_P$ be its
unipotent radical, and let $L = L_P$ be the Levi subgroup of $P$ containing $T$, so
that $P$ is a semi-direct product of $L$ and $U$.  Then $B_L = B \cap L$ is a Borel subgroup
of $L$.  We denote the Lie algebras of the groups $G$, $B$, $P$, $U$, $L$, $B_L$
by $\mathfrak{g}$, $\mathfrak{b}$, $\mathfrak{p}$, $\mathfrak{u}$, 
$\mathfrak{l}$, and $\mathfrak{b}_L$, respectively, and we denote by $\mathfrak{h}$ the Cartan subalgebra of
$\mathfrak{g}$ corresponding to $T$.

Let $R \subseteq \mathfrak{h}^*$ be the set of roots of $\mathfrak{g}$ with respect to
the Cartan algebra $\mathfrak{h}$, and let $R^+$ be the set of positive roots (fixed 
by the choice of Borel subgroup).  Let $\Delta = \{\alpha_1, \ldots, \alpha_r\}$ be the set of simple roots,
where $r$ is the rank of $G$.  Define the elements $x_i \in \mathfrak{h}$ by the equations $\alpha_i(x_j) = \delta_{ij}$
 We denote the Killing form on $\mathfrak{h}$ and $\mathfrak{h}^*$
using angle brackets $\langle,\rangle$; assume it is normalized so that $\langle \theta, \theta \rangle = 2$,
where $\theta$ is the highest root.  The isomorphisms $\mathfrak{h} \rightarrow \mathfrak{h}^*$ 
$\mathfrak{h} \rightarrow \mathfrak{h}^*$ induced by the Killing form we will denote by $\kappa$.
Define the coroots $\alpha_1^\vee, \ldots, \alpha_r^\vee \in \mathfrak{h}^*$
as $\alpha_i^\vee = \frac{2\alpha_i}{\langle \alpha_i, \alpha_i \rangle}$, and the fundamental
weights $\omega_i \in \mathfrak{h}^*$ by $\langle \omega_i, \alpha_j \rangle = \delta_{ij}$.
Let $\mathfrak{h}_+$ and $\mathfrak{h}_+^*$ be the dominant Weyl chambers, and 
$\Lambda_+ \subseteq \mathfrak{h}_+^*$ be the semigroup of dominant integral weights.
For a parabolic $P$, let $\Delta_P$ be the associated set of simple roots, and let $R_L$
be the set of roots of $\mathfrak{l}$ with respect to $\mathfrak{h}$.  If $P$ is maximal,
we will usually denote the excluded root and fundamental weight as $\alpha_P$ and $\omega_P$, 
respectively.  Let $U_\alpha$ denote the root group in $G$ associated to $\alpha$.

For any positive root $\beta$ we will denote the associated reflection by $s_\beta$.  Let $\{s_1, \ldots, s_r\} \in W$ 
be the generating set of simple reflections.  For a parabolic subgroup $P$ let
$W_P$ be the associated Weyl group, which is also the Weyl group of $L$.  For every coset in
$W/W_P$ there is a unique minimal length representative.  Let $W^P$ be the set of minimal representatives.
For any $w \in W$, we denote by $l(w)$ its length.  The symbols $<,>,\leq,\geq$ will denote the
Bruhat ordering in $W$.
For any $u \in W^P$ let $C_u = BuP \subseteq G/P$ be the Schubert cell associated to $u$ and 
$X_u = \overline{BuP}$ the associated Schubert variety.  Let $Z_u \subseteq X_u$ be the nonsingular
locus.  We denote by $\sigma_u \in \text{H}^0(G/P)$ the Poincar\'{e} dual of the homology class 
associated to $X_u$.

Let $\lambda_1, \ldots, \lambda_n$ be dominant integral weights and $\ell$ be a positive integer.  We call the tuple 
$\vec{w} = (\lambda_1, \ldots, \lambda_n, \ell)$ the Kac-Moody weight data associated to $\lambda_1, \ldots, \lambda_n$
and $\ell$.  The integer $\ell$ is will be called the level. Given a weight $\lambda$ and a level
$\ell$ such that $\langle \lambda, \theta \rangle \leq \ell$, we can associate a point of the fundamental alcove
$\mu = \frac{\kappa(\lambda)}{\ell}$.  We say that $\vec{w}$ is in
the multiplicative polytope if $\vec{\mu} = (\frac{\kappa(\lambda_1)}{\ell}, \ldots, \frac{\kappa(\lambda_n)}{\ell})$ is
in the multiplicative polytope.

\subsubsection{Acknowledgements}
I am grateful to Prakash Belkale for his help and encouragement in the 
preparation of this work.  I also thank Shrawan Kumar for helpful discussions about loop groups; Vikraman Balaji 
for useful clarifications on parahoric group schemes; and Angela Gibney and Swarnava Mukhopadhyay for helpful comments about 
the exposition, and potential applications to conformal blocks divisors.  This work forms part of the author's
Ph.D. thesis at the University of North Carolina, Chapel Hill.

\section{Preliminaries on the multiplicative polytope and conformal blocks}\label{PolySec}

In this section we discuss in more detail the main objects of interest in this paper.  First we recall how the quantum 
cohomology of generalized Grassmannians $G/P$ parametrize
the faces of the polytope.  Then we give a definition of conformal blocks as spaces of coinvariants,
and describe how conformal blocks can be identified with space of generalized theta functions.  The methods
we use to prove the reduction theorem rely on this identification, and for the bulk of the proof we 
will be working with conformal blocks as generalized theta functions.

\subsection{The multiplicative polytope and quantum cohomology}\label{MULTQUANT}

In this section we explain in more detail the inequalities
defining its faces of the multiplicative polytope.  As part of describing the faces we will also discuss
the small quantum cohomology ring of the flag varieties $G/P$.

Since the multiplicative polytope is a convex polyhedron, it is defined by a unique set of irredundant linear inequalities.
The inequalities are parametrized by products in the small quantum cohomology ring of the flag varieties $G/P$, where $P$ is 
a maximal parabolic.  In type A these are the complex Grassmannians.  Let us begin by reviewing the general type
combinatorics of the cohomology of $G/P$.

For any flag variety $G/P$ there is a canonical cell decomposition into \emph{Schubert cells}.  The Schubert cells
are parametrized by cosets in $W/W_P$, where $W$ is the Weyl group of $G$, and $W_P$ is the Weyl group
of $P$.  These cosets each have a unique (minimal length) representative, the set of which
is denoted $W^P$. We denote by $C_w$ the Schubert cell corresponding to $w \in W^P$, and by $\sigma_w \in \text{H}^*(G/P)$
the Poincar\'{e} dual of the homology class of $C_w$.  It is well known that the cohomology
ring $\text{H}^*(G/P)$ is generated by the Schubert classes $\sigma_w$, and therefore the cohomology ring
is determined by the set of positive numbers $c_{u,v}^w$ such that
\[
\sigma_{u} \cdot \sigma_{v} = \sum_{w \in W^P } c_{u,v}^{w} \cdot \sigma_w^*.
\]

The small quantum cohomology ring of $G/P$ is defined as follows.  Let $X = \mathbb{P}^1$ and fix 3 distinct points in
$X$: say $p_1, p_2, p_3 \in X$.  Then for any $d \geq 0$ and $w_1, w_2, w_3 \in W^P$, the \emph{Gromov-Witten invariant}
$\langle \sigma_{w_1}, \sigma_{w_2}, \sigma_{w_3} \rangle_d$ is defined as the number of degree $d$ maps $f:X \rightarrow G/P$
such that $f(p_i) \in g_i C_{w_i}$ for generic $g_i \in G$, where the invariant is zero in the case that there 
are an infinite number of such maps.  Additively the quantum cohomology ring is just $\text{QH}^*(G/P) = \text{H}^*(G/P)[q]$, where
$q$ is an indeterminant.  The quantum product is then defined in terms of the Gromov-Witten invariants:
\[
\sigma_{w_1} \ast \sigma_{w_2} = \sum_{\underset{w_3 \in W^P}{d \geq 0} } q^d\langle \sigma_{w_1},\sigma_{w_2},\sigma_{w_3} \rangle_d \cdot \sigma_{w_3}^*
\]
where $\sigma_w^*$ is the unique class such that $\sigma_w \cdot \sigma_w^* = \text{[pt]}$. Since the image of a degree zero map
$f: X \rightarrow G/P$ is just a point,  by the uniqueness of $\sigma_w^*$ it follows that 
$\langle \sigma_{w_1},\sigma_{w_2},\sigma_{w_3} \rangle_0 = c_{w_1,w_2}^{w_3}$.

Remarkably, the multiplicative polytope is determined by inequalities parametrized by products in $\text{QH}^*(G/P)$.
In \cite{TWPB}, Teleman and Woodward proved the following, building on the work of Agnihotri, Woodward, and Belkale
for the group $\text{SL}_{r+1}$ \cite{AW98,BEL01}.

\begin{thm}\cite{TWPB}
For any maximal parabolic $P \subseteq G$, and any product
$\sigma_{w_1} \ast \cdots \ast \sigma_{w_n} = q^d \text{[pt]}$, a point $\vec{\mu} \in \Delta_n(G)$ must satisfy
\[
\sum_{i=1}^n \omega_P(w_i^{-1} \mu_i) \leq d,
\]
where $\omega_P$ is the fundamental weight associated to $P$.
The multiplicative polytope $\Delta_n(G)$ as a subset of the $n$-fold alcove $\mathcal{A}^n$ is determined by the above
inequalities.
\end{thm} 
To facilitate discussion about the geometry of the multiplicative polytope, we make the following definitions.

\begin{defin}
The linear faces of $\mathcal{A}$ are called the \emph{(Weyl) chamber walls}, and the affine face is called
the \emph{alcove wall}.  A point in $\mathcal{A}^n$ is on a chamber or alcove wall if at least one $\mu_i$ in 
$(\mu_1, \ldots, \mu_n)$ is on a chamber or alcove wall.  A \emph{regular face} of $\Delta_n(G)$ is a face that 
intersects the interior of $\mathcal{A}^n$; a \emph{TW-face} is a face corresponding to a product 
$\sigma_{w_1} \ast \cdots \ast \sigma_{w_n} = q^d \text{[pt]}$ as above.  If a face of $\Delta_n(G)$ corresponds to a product 
$\sigma_{w_1} \ast \cdots \ast \sigma_{w_n} = q^d \text{[pt]}$ in $\text{QH}^*(G/P)$, we say that the 
face is \emph{degree $d$} and \emph{of type $P$}.
\end{defin}

While the above inequalities indeed determine the multiplicative polytope, they are not irredundant.  
The facets (codimension-one faces) intersecting the interior of $\mathcal{A}^n$ correspond to a subset of the
TW-faces. The subset of irredundant inequalities can be selected by replacing the quantum product
$\ast$ with the quantum-BK-product $\circledast_0$.
Belkale and Kumar first defined a degeneration of the ordinary cohomology product in \cite{BK06}, 
and proved that the reduced set of inequalities are sufficient to define
the additive eigencone.  Ressayre then proved that these inequalities are irredundant in \cite{RSS10}.
Finally, Belkale and Kumar built upon their work and Ressayre's to define a new quantum product $\circledast_0$
that gives rise to an irredundant set of inequalities for the multiplicative polytope \cite{BK}.

We will not need the product $\circledast_0$ for the reduction theorem, which holds even on TW-faces that
are not regular facets.  The regular facets are however the most accessible source of examples for the
theorem, since it is not clear in general which products correspond to planes intersecting the multiplicative polytope
non-trivially.
Lastly, we note that Belkale and Kumar prove that the product $\circledast_0$ coincides with the usual quantum product when
$G/P$ is \emph{cominiscule}, which includes all Grassmannians and Lagrangian Grassmannians, among
other flag varieties.  In general, a flag variety $G/P$ is cominiscule if $\alpha_P$ appears with coefficient
$1$ in the highest root $\theta$ of $G$.

\subsection{Conformal blocks}\label{CBDef}

In this section we define the main objects of interest in this paper: spaces and bundles of conformal blocks.
We start by describing the construction of spaces of conformal blocks as certain invariant
spaces of representations of infinite dimensional algebras related to the untwisted affine Kac-Moody
algebra associated to $G$.  We also describe the connection between conformal blocks and the multiplicative
polytope.  Finally, we introduce principal and parabolic $G$-bundles, and 
give an alternative definition of spaces of conformal blocks as spaces of generalized theta functions.

\subsubsection{Conformal blocks as spaces of coinvariants}

Let $X$ be a smooth projective and connected algebraic curve over $k=\mathbb{C}$.  We now give the definition
of the space of conformal blocks over $X$ in terms of representations of an infinite dimensional lie algebra $\hat{\mathfrak{g}}$.  
We will not use this definition in the rest of the paper.  For more details on this definition see \cite{FUSION};
for a more comprehensive treatment of conformal blocks see \cite{UENO}; for background on Kac-Moody algebras, see
\cite{KACINF}.

Let $K=\mathbb{C}((z))$ be the field of formal Laurent series with complex coefficients.  Let 
$\hat{\mathfrak{g}} = (\mathfrak{g} \otimes K) \oplus \mathbb{C}\cdot c$. The bracket for $\hat{\mathfrak{g}}$
is given by
\[
[X \otimes f, Y \otimes g] = [X,Y] \otimes fg + c \cdot \langle X, Y \rangle \text{Res}(g \cdot \text{d}f)
\]
where $X,Y \in \mathfrak{g}$, $f, g \in K$, the product $\langle, \rangle$ is the normalized Killing form,
and $\text{Res}(g \cdot \text{d}f)$ is the residue of $g \cdot \text{d}f$.  The vector $c$ is central.  This Lie algebra is
a subalgebra of the completion of the untwisted affine Kac-Moody Lie algebra associated to $G$.
Given a dominant integral weight $\lambda$, and an integer $\ell$ such that $\ell(\lambda) \leq \ell$,
there is a unique associated simple $\hat{\mathfrak{g}}$-module $\mathcal{H}_{\lambda, \ell}$.
Given a tuple of weights $(\lambda_1, \ldots, \lambda_n)$ and a level $\ell$, we write
$\mathcal{H}_{\vec{\lambda}, \ell} = \mathcal{H}_{\lambda_1, \ell} \otimes \cdots \otimes \mathcal{H}_{\lambda_n, \ell}$.

The curve $X$, along with a collection of distinct points $p_1, \ldots, p_n \in X$ determines an action 
on $\mathcal{H}_{\vec{\lambda}}$ as follows.  Let $X^* = X \setminus \{p_1, \ldots, p_n\}$.  Then
any function $f \in \mathcal{O}(X^*)$ determines $f_1, \ldots, f_n \in K$ by taking its Laurent
series at each point.  Thus given $X \otimes f \in \mathfrak{g}(X^*)$, we get an element $X \otimes f_i$ of the loop
algebra $\mathfrak{g} \otimes K \subseteq \hat{\mathfrak{g}}$ for each $i$.  The action of $\mathfrak{g}(X^*)$
on $\mathcal{H}_{\vec{\lambda}}$ is given in the obvious way:
\[
(X \otimes f)(v_1 \otimes \cdots \otimes v_n) = \sum_{i=1}^n v_1 \otimes \cdots \otimes (X \otimes f_i) v_i \otimes \cdots \otimes v_n.
\]

Then the space of conformal blocks is defined as follows.
\begin{defin}
For any tuple of weights $(\lambda_1, \ldots, \lambda_n)$ and level $\ell$,
the \emph{space of conformal blocks} $\mathcal{V}^{\dagger}_{\vec{\lambda}, \ell}(X, \vec{p})$ is defined as
\[
\mathcal{V}^{\dagger}_{\vec{\lambda}, \ell}(X, \vec{p}) = \text{Hom}_{\mathfrak{g}(X^*)}(\mathcal{H}_{\vec{\lambda}}, \mathbb{C})
\]
where $\mathbb{C}$ has the trivial $\mathfrak{g}(X^*)$-action.  We denote the dual of this space as 
$\mathcal{V}_{\vec{\lambda}, \ell}(X, \vec{p})$.
\end{defin}

It is not obvious from the above definition, but the space of conformal blocks is in fact \emph{finite dimensional}.
Furthermore, the space of conformal blocks depends on the choice of points $p_1, \ldots, p_n \in X$, but remarkably its dimension
does not.  In fact the above definition works for families of pointed curves, and even allows degeneration to \emph{stable pointed curves},
leading to the following definition.
\begin{defin}
For any tuple of weights $(\lambda_1, \ldots, \lambda_n)$, level $\ell$, and genus $g$,
the \emph{bundle of conformal blocks} $\mathbb{V}_{\vec{\lambda}, \ell}$ is the vector bundle over the moduli stack
$\overline{\mathcal{M}}_{g,n}$ of genus $g$ stable pointed curves with $n$ marked points,
such that the fiber over $(X,\vec{p}) \in \overline{\mathcal{M}}_{g,n}$ is $\mathcal{V}_{\vec{\lambda}, \ell}(X, \vec{p})$.
\end{defin}
For the construction of the sheaf of conformal blocks for families of stable pointed curves, and for the proof
that this sheaf is of finite rank and locally free over $\overline{\mathcal{M}}_{g,n}$, see \cite{UENO}.

\subsubsection{Conformal blocks as generalized theta functions}\label{CBDEF}

Let $X$ be a scheme over $k$.  Then a \emph{principal G-bundle over $X$} is a $G$-scheme (with a right $G$-action) 
$\mathcal{E}$ together with a projection morphism $\pi : \mathcal{E} \rightarrow X$ such that $\mathcal{E}$ is 
locally trivial in the \'{e}tale topology.  In other words, there is a surjective \'{e}tale cover $S \rightarrow X$ 
such that the pullback $\mathcal{E}_{|S}$ over $S$ is isomorphic as a $G$-scheme to the trivial $G$-scheme $S \times G$.

Now fix a smooth, projective, and connected curve $X$ over $k$.  A \emph{family of principal $G$-bundles over $X$}
is simply a principal $G$-bundle over $X \times S$, for any $k$-scheme $S$.  The category $\text{Bun}_G$ of families of
principal $G$-bundles over $X$ forms a smooth algebraic stack over $k$.  For an introduction to the moduli stack
of $G$-bundles see Sorger's notes \cite{SOR}; for detailed proofs of some of its basic geometric properties see
Wang's senior thesis \cite{WAN}.

Now let $\vec{p} = (p_1, \ldots, p_n)$ be distinct closed points of $X$, and assume that $G$ is semisimple, with the
associated notations and conventions described in the introduction. 
\begin{defin} A \emph{quasi-parabolic $G$-bundle $\widetilde{\mathcal{E}} = (\mathcal{E},\vec{g})$ over $X$ (with full flags)} is a principal $G$-bundle 
$\mathcal{E} \rightarrow X$ together with choices of flags
$\overline{g}_i \in (\mathcal{E}/B)_{p_i}$.   A \emph{family} of quasi-parabolic bundles parametrized by $S$ is a
principal $G$-bundle $\mathcal{E} \rightarrow X \times S$ together with sections $\overline{g}_i$ of $\mathcal{E}_{|p_i}/B \rightarrow S$
for each $i$.  We denote the moduli stack of quasi-parabolic bundles $\normtext{Parbun}_G(X, \vec{p})$,
or more concisely as $\normtext{Parbun}_G$ when $X$ and $\vec{p}$ are understood.
\end{defin}

We will often abuse terminology and call quasi-parabolic bundles simply
parabolic bundles.  The phrase ``full flags" indicates that the flags are elements of a fiber of $\mathcal{E}$ modulo
the Borel $B$, as opposed to ``partial flags", which would be elements of a fiber modulo a parabolic subgroup 
$Q \supset B$. The stack $\text{Parbun}_G$ is again a smooth Artin stack, since the forgetful morphism
$\text{Parbun}_G \rightarrow \text{Bun}_G$ is smooth.

Our methods rely on the fact that conformal blocks can be identified with spaces of \emph{generalized theta functions}.  
Let $X(B)$ denote the character group of $B$.  Then we have the following:

\begin{thm}[{\cite{LASOR, SOR99}}]
For any simple, simply-connected algebraic group $G$, there is a line bundle $\mathcal{L}$ on $\text{Parbun}_G$ such that
\[
\normtext{Pic}(\normtext{Parbun}_G) \cong \mathbb{Z}\mathcal{L} \times \prod_{i=1}^n{X(B)}.
\]
\end{thm}
\begin{rmk}
In types $A$ and $C$, $\mathcal{L}$ is a determinant of cohomology line bundle; in types $B$, $D$, and for $G=G_2$, 
$\mathcal{L}$ is the Pffafian line bundle, a canonical square root of a given determinant line bundle.
\end{rmk}

By identifying weights with characters, Kac-Moody weight data $\vec{w} = (\lambda_1, \ldots, \lambda_n, \ell)$ corresponds to a line bundle 
$\mathcal{L}_{\vec{w}} = \mathcal{L}^\ell \otimes \mathcal{L}_{\lambda_1} \otimes \cdots \otimes \mathcal{L}_{\lambda_n}$ 
over $\text{Parbun}_G$.  Suppose that we have a parabolic bundle $\mathcal{E}$ with full flags 
$\overline{g}_1, \ldots, \overline{g}_n$, corresponding to a point in $\text{Parbun}_G$.  Then the fiber of 
$\mathcal{L}_{\lambda_i}$ is the fiber over $\overline{g}_i$ of $\mathcal{E} \times^B \chi_i \rightarrow \mathcal{E}/B$,
where $\chi_i$ is the character of $B$ corresponding to $\lambda_i$.  

The line bundle $\mathcal{L}$ is a canonical root of a determinant of cohomology line bundle.
Let $V$ be a finite dimensional irreducible representation of $G$.
Let $\mathcal{E}$ be a principal $G$-bundle over $X$, and $\mathcal{E}(V)$ the associated vector bundle.  Then we
make the following definition.

\begin{defin}
The \emph{determinant of cohomology} line bundle $D(V)$ over $\text{Parbun}_G$ associated to the representation $V$ 
is the line bundle whose fiber over $(\mathcal{E}, \overline{g}_1, \ldots, \overline{g}_n)$ is 
\[
\left( \bigwedge^{\text{max}} \text{H}^0(X, \mathcal{E}(V)) \right)^* \otimes \bigwedge^{\text{max}} \text{H}^1(X, \mathcal{E}(V)).
\]
\end{defin}

The bundle $D(V)$ can be identified in $\text{Pic}(\text{Parbun}_G)$ as the line bundle corresponding to trivial characters
and a level $\ell$ equal to the \emph{Dynkin index} of $V$: 
let $f: \mathfrak{g}_1 \rightarrow \mathfrak{g}$ be an embedding of simple Lie algebras, and assume that
the Killing forms $\langle, \rangle_1$ and $\langle, \rangle$ of the algebras are normalized so that
$\langle \theta_1, \theta_1 \rangle_1 = \langle \theta, \theta \rangle =2$, where $\theta_1, \theta$
are the highest roots of $\mathfrak{g}_1$, $\mathfrak{g}$, respectively.
Then there is a unique integer $m_f$ (the Dykin index) such that for any $x, y \in \mathfrak{g}$, 
\[
\langle f(x), f(y) \rangle = m_f \langle x, y \rangle_1.
\]
For a faithful representation $V$ of $\mathfrak{g}$, the Dynkin index is defined as the Dynkin index of 
$\mathfrak{g} \rightarrow \mathfrak{sl}(V)$.
See Theorem 5.4 in \cite{SNR} and section 6 of \cite{LASOR} for a proof of the identification of the
level of $D(V)$.

Our main interest in parabolic bundles is that the global section of the line bundles $\mathcal{L}_{\vec{w}}$ can be
identified with spaces of conformal block.  The following theorem was proven in the form we need by Laszlo and Sorger in \cite{LASOR}.

\begin{thm}{\cite{LASOR}}
Given Kac-Moody weight data $\vec{w}$,  the space of global sections 
$\normtext{H}^0(\normtext{Parbun}_G, \mathcal{L}_{\vec{w}})$ is naturally isomorphic to the
space of conformal blocks $\mathcal{V}^{\dagger}_{\mathfrak{g}, \vec{w}}(X, \vec{p})$.
\end{thm}

Finally, conformal blocks in genus zero are connected to the multiplicative polytope in the following way.  
This theorem follows from the description of conformal blocks
as generalized theta functions discussed below, and Theorem 5.2 in \cite{BK}.

\begin{thm}
If $X \cong \mathbb{P}^1$, then for any tuple of weights $(\lambda_1, \ldots, \lambda_n)$ and level $\ell$, there exists an integer $N>0$ such that
$\normtext{dim}(\mathcal{V}^{\dagger}_{\vec{N\lambda}, N\ell}(X, \vec{p})) > 0$ if and only if 
$(\frac{\kappa(\lambda_1)}{\ell}, \ldots, \frac{\kappa(\lambda_n)}{\ell}) \in \mathcal{A}^n$ lies in the multiplicative
polytope.
\end{thm}

\section{Parahoric bundles}\label{ParahoricSec}
The goal of this section is to prove some basic results about line bundles and their global sections on
stacks of parahoric bundles.  The main fact that we need is that conformal blocks descend to moduli stacks 
of parahoric bundles.  We need to work with parahoric bundles to do the properness calculation in section 
\ref{PReductionSec}, and we need conformal blocks to descend to complete the proof of the reduction theorem
in section \ref{ProofSec}.

In section \ref{Moduli} we begin with a brief discussion of moduli spaces of parabolic $G$-bundles,
which will serve to motivate the introduction of parahoric bundles.  In the section \ref{ParahBun}
we introduce parahoric bundles and their moduli stacks, and review the main results of Balaji
and Seshadri in \cite{BAS}, where they show that these stacks can be identified with stacks 
of equivariant $G$-bundles over a Galois cover of our curve $Y \rightarrow X$.  In section
\ref{flags} we study the special fibers of parahoric group schemes and the relative flag structure
of stacks of parahoric bundles.  Finally, in section \ref{PProp}, we prove that conformal
blocks descend to stacks of parahoric bundles.

Throughout this section, $G$ is a semisimple, connected and simply-connected algebraic group
over $k = \mathbb{C}$, and $X$ is a smooth, projective and connected curve over $k$ of
arbitrary genus.

\subsection{Moduli spaces of parabolic bundles}\label{Moduli}

The stack of parabolic bundles $\text{Parbun}_G$, while algebraic, smooth and irreducible, is not proper, or even
separated.  However, it is possible to construct projective moduli spaces of parabolic bundles.
Moduli spaces $\mathbb{M}_{\vec{w}}$ of parabolic bundles depend on weight data $\vec{w} = (\lambda_1, \ldots, \lambda_n, \ell)$.
If the weight data corresponds to an interior point of $\mathcal{A}^n$, then the $k$-points
of this moduli space correspond to \emph{grade equivalence classes} of \emph{semistable} parabolic $G$-bundles
with flags in $G/B$.  The locus of semistable parabolic bundles with respect to $\vec{w}$
is defined as the set of bundles $\widetilde{\mathcal{E}} \in \text{Parbun}_G$ such that there exists
a section $s \in \text{H}^0(\text{Parbun}_G, \mathcal{L}_{\vec{w}}^N)$ for some $N$ such that $s(\widetilde{\mathcal{E}}) \neq 0$.
Grade equivalence identifies bundles that must be identified in any separated moduli space of $G$-bundles; for the
definition of grade equivalence see \cite{TWPB}.

Alternatively, semistability can be defined in terms of $P$-reductions.  Let $P \supseteq B$ be 
a maximal parabolic.  Consider a 
parabolic $G$-bundle $\widetilde{\mathcal{E}} = (\mathcal{E}, \overline{g}_1, \ldots, \overline{g}_n) \in \text{Parbun}_G$ and a $P$-reduction 
$\phi: X \rightarrow \mathcal{E}/P$.  We can trivialize $\mathcal{E}$ over an open set $U$ containing the points 
$p_1, \ldots, p_n$, and then clearly there are unique Weyl group elements $w_1, \ldots, w_n \in W^P$ such that 
$\phi(p_i) \in \tilde{g}_i C_{w_i}$.  The $w_i$'s do not depend on the trivialization and together are called the \emph{relative position}
of the $P$-reduction in $\widetilde{\mathcal{E}}$.  Then semistability of $\widetilde{\mathcal{E}}$ is defined in
\cite{TWPB} as follows.
We say that 
$(\mathcal{E}, \overline{g}_1, \ldots, \overline{g}_n)$ is \emph{semistable} if for every maximal parabolic $P$ and
every $P$-reduction the following inequality is satisfied
\[
\sum_{i=1}^n \langle \omega_P, w_i^{-1}\lambda_i \rangle \leq \ell d,
\]
where $w_1, \ldots, w_n \in W^P$ give the relative position of the reduction in $\widetilde{\mathcal{E}}$,
and $d$ is the degree of the reduction.  The bundle is \emph{stable} if strict inequality is satisfied for every $P$-reduction.
We say that $\sum_{i=1}^n \langle \omega_P, w_i^{-1}\lambda_i \rangle - \ell d$ is the \emph{parabolic degree} of the
$P$-reduction.
For a proof that these two definitions of semistability are equivalent, see Proposition \ref{SEMIEQ}.

\subsubsection{Moduli spaces on the boundary of $\mathcal{A}$}

When one or more of the weights are on a chamber wall, we can still construct a moduli space $\mathbb{M}_{\vec{w}}$,
however it is too small to be a moduli space of parabolic bundles with full flags.
Instead $\mathbb{M}_{\vec{w}}$ is a  moduli space of parabolic bundles with partial flags in $G/Q$, for some $Q \supseteq B$.  

A weight $\lambda$ corresponds to a standard parabolic $Q \subseteq G$ in the following way: if $\Delta_Q$
is the set of simple roots $\alpha$ such that $\langle \lambda , \alpha \rangle = 0$, then $Q$ is the
parabolic corresponding to $\Delta_Q$.  So given weights
$\lambda_1, \ldots, \lambda_n$ we get parabolics $Q_1, \ldots, Q_n$.  Then we define $\text{Parbun}_G(\vec{Q})$
to be the moduli stack of principal $G$-bundles over $X$ along with choices of flags $\overline{g}_i \in \mathcal{E}_{|p_i}/Q_i$
for each $i$.  The space $\mathbb{M}_{\vec{w}}$ is the moduli space of grade equivalence class of semistable
parabolic bundles in $\text{Parbun}_G(\vec{Q})$.

When one or more weight is on the alcove wall, i.e. if $\langle \lambda_i, \theta \rangle = \ell$, then 
$\mathbb{M}_{\vec{w}}$ identifies parabolic bundles in a similar way, but will identify bundles with
different underlying principal $G$-bundles.  In this case $\mathbb{M}_{\vec{w}}$ is naturally a moduli
space of \emph{parahoric bundles}.  Parahoric bundles are by definition torsors over a smooth
group scheme $\mathcal{G}$ over $X$ associated to \emph{parahoric subgroups} $\mathcal{P}_1, \ldots, \mathcal{P}_n$,
which in turn are determined by our choice of weight data $\vec{w}$.  

There is a natural morphism $\text{Parbun}_G \rightarrow \text{Bun}_{\mathcal{G}}$, where $\text{Bun}_{\mathcal{G}}$
is the stack of parahoric bundles.  The rest of the section is devoted to introducing 
parahoric bundles and studying this morphism.  Specifically, we will show that it is exactly analogous to
the quotient morphism $\text{Parbun}_G \rightarrow \text{Parbun}_G(\vec{Q})$, and use this description
to show that conformal blocks descend to $\text{Bun}_{\mathcal{G}}$.

\subsection{Parahoric bundles}\label{ParahBun}

Parahoric bundles are torsors over parahoric group schemes over $X$
which are generically the trivial group scheme $X^* \times G$, but near each $p_i$ are smooth group schemes originally
arising in Bruhat-Tits theory \cite{BTII}.  The parabolic bundles above 
can be identified with parahoric bundles, but there are parahoric bundles that do not have 
an underlying principal $G$-bundle and therefore cannot be described as parabolic bundles.  

Our primary technique for working with parahoric bundles is to work with associated equivariant bundles over
a ramified extension $Y \rightarrow X$, following the work of Balaji and Seshadri in \cite{BAS}; in this way we can view 
parahoric bundles as ``orbifold bundles", with equivariant bundles acting as orbifold charts.  We use this
description of parahoric group schemes and bundles to describe the closed fibers of the group schemes over each $p_i$.

\subsubsection{Basic definitions}

We mostly follow the notation and conventions of Tits in \cite{TITS}.  
Let $K = k((z))$ and $A = k[[z]]$. Then $K$ is naturally a local field with valuation $\nu$.
Assume that $G$ is a connected, simply connected, semi-simple group over $k$.  We will denote the associated split group over $K$ as $G(K)$. 
Choose a maximal torus $T$ of $G$, and let $X^*(T(K)) = \text{Hom}_K(T(K), \mathbb{G}_m^K)$ denote the group of $K$-valued characters of $T$, 
and $X_*(T(K)) = \text{Hom}_K(\mathbb{G}_m^K, T(K))$ the group of cocharacters.  Let  $V = X_*(T(K)) \otimes \mathbb{R}$.  Then an 
\emph{affine root} $\alpha + k$ is an affine function on $V$ given by a root $\alpha \in R$ and an integer $k$ (it
will be clear whether $k$ is an integer or a field in context).  The vector space $V$ acts on the \emph{apartment} $A(T(K))$
associated to $T(K)$, making $A(T(K))$ an affine space.  A choice of origin in $A(T(K))$ allows us to identify $A(T(K))$ with $V$,
which we fix from now on.  For every affine root $\alpha +k$, there is an associated \emph{half-apartment} $A_{\alpha + k}$
defined as $A_{\alpha + k} = (\alpha + k)^{-1}([0, \infty))$, with its boundary denoted $\delta A_{\alpha + k}$.
The \emph{chambers} of $A(T(K))$ are the connected components of the complement of all the walls $\delta A_{\alpha + k}$.
When $G$ is simple, the chambers are simplices and the fundamental alcove $\mathcal{A}$ is identified with the chamber bounded
by the walls corresponding to the simple roots $\alpha_1, \ldots, \alpha_r$, and the affine root $\theta - 1$.
When $G$ is semisimple the chambers are polysimplices, and when $G$ is not semisimple, the chambers are products 
of polysimplices and real affine spaces.  

The \emph{Bruhat-Tits building} $\mathcal{B}(G(K))$ is a space constructed by gluing together the apartments associated to each torus.
Associated to each affine root $\alpha + k$ is a subgroup $X_{\alpha + k}$ of $U_\alpha(K)$: the choice of origin of $A(T(K))$
determines an isomorphism $U_\alpha(K) \cong \mathbb{G}_{a,K}$, and $X_{\alpha + k}$ is defined as $\nu^{-1}([k, \infty))$ in
$U_\alpha(K)$ with respect to this isomorphism, which justifies writing $X_{\alpha + k}$ as $U_\alpha(z^{-k}A)$.  
Then the building $\mathcal{B}(G(K))$ has a $G(K)$-action such that $X_{\alpha + k} \cong U_\alpha(z^{-k}A)$ 
fixes the half-apartment $A_{\alpha + k}$ pointwise.  Furthermore $\mathcal{B}(G(K))$ is the union of $g A(T(K))$ for
$g \in G(K)$, and the normalizer $N(K)$ of $T(K)$ fixes $A(T(K))$.

For simplicity assume $G$ is simple.  Just like we associate a parabolic subgroup to a weight $\lambda$, we can
associate a \emph{parahoric subgroup} $\mathcal{P}$ of $G(K)$ to each pair $(\lambda, \ell)$ such that 
$\ell(\lambda) = \langle \lambda, \theta \rangle \leq \ell$.

\begin{defin}
Let $\lambda$ be a dominant integral weight, and $\ell$ be a level, corresponding to a point $\mu = \frac{1}{\ell}\kappa(\lambda)$ of the
fundamental alcove.  The point $\mu$ lies in the interior of a unique face $F$ of the building $\mathcal{B}(G(K))$.
Then the \emph{parahoric subgroup associated to $(\lambda, \ell)$} is defined as the stabilizer $\mathcal{P}$ in $G(K)$ of $F$.
Alternatively, $\mathcal{P}$ is generated as a subgroup as follows:
\[
\mathcal{P} = \langle T(A), U_\alpha(z^{-k}A) \mid \mu \in A_{\alpha + k} \rangle.
\]
\end{defin} 
\begin{rmk}
Letting $G(A') = G(k[[z^{1/m}]])$ for some integer $m$, Balaji and Sesahdri showed that $\mathcal{P}$ can be identified 
(non-canonically) with an invariant subgroup of $G(A')$ under an action by a finite cyclic group $\Gamma$ \cite{BAS}.  
We will make this explicit in section \ref{FIBERS}.
\end{rmk}
\begin{rmk}
Parahoric subgroups also correspond to subsets of the vertices $\{ v_0, \ldots, v_r \}$
 of the affine Dynkin diagram of $G$, with the empty set corresponding
to an Iwahori subgroup $\mathcal{I}$.  The Iwahori subgroup corresponding to $B$ is defined as the inverse
image of $B$ with respect to the evaluation map $ev_0 : G(A) \rightarrow G$.  Each standard parabolic $P$
corresponds in the same way to a parahoric subgroup contained in $G(A)$, with vertex set the same as $P$.
The vertex set of $G(A)$ is the set of all vertices $v_1, \ldots, v_r$ of the finite Dynkin diagram.
The bijection between vertex sets and parahoric subgroups is inclusion preserving, and so
parahoric subgroups corresponding to vertex sets containing the vertex $v_0$ are not contained in $G(A)$.
The vertex set corresponding to $(\lambda, \ell)$ is the subset of $\{v_1, \ldots, v_r\}$ corresponding
to simple roots $\alpha_i$ such that $\langle \lambda, \alpha_i \rangle = 0$, adding $v_0$ if in addition
$\langle \lambda, \theta \rangle = \ell$.
\end{rmk}
\begin{rmk}
When $G$ is semisimple, the definition of a parahoric subgroup is exactly the same.
However in this case there is a highest weight $\theta_i$ and level $\ell_i$ for each factor
of the Dynkin diagram of $G$, making the identification of weight data $\vec{w}$ and a
point in the alcove more complicated.
\end{rmk}

\subsubsection{Parahoric group schemes, bundles, and associated loop groups}

One of the main results of \cite{BTII} is the existence of a group scheme $\mathcal{G}$, smooth over $\text{Spec}(A)$, such
that $\mathcal{G}(K) \cong G(K)$ and $\mathcal{G}(A) \cong \mathcal{P}$, for any parahoric $\mathcal{P}$.  These
group schemes are \emph{\'{e}ttof\'{e}}, which means the following: given any $A$-scheme $\mathcal{N}$
and $K$-morphism $u_K: \mathcal{G}_K \rightarrow \mathcal{N}_K$ such that $u(\mathcal{G}(A)) \subseteq \mathcal{N}(A)$,
there is a unique extension to an $A$-morphism $u: \mathcal{G} \rightarrow \mathcal{N}$.  This implies the 
uniqueness of $\mathcal{G}$ up to unique isomorphism.

Then we have the following definitions.

\begin{defin}
To weight data $\vec{w} = (\lambda_1, \ldots, \lambda_n, \ell)$ we associate a smooth group scheme $\mathcal{G}$ over $X$,
which is the trivial group scheme $X^* \times G$ over $X^* = X \setminus \{p_1, \ldots, p_n\}$, and in a formal
neighborhood of each $p_i$ is isomorphic to the parahoric group scheme associated to each $(\lambda_i, \ell)$.  A
\emph{parahoric $\mathcal{G}$-bundle} is simply a $\mathcal{G}$-torsor; that is, a scheme over $X$ with a right $\mathcal{G}$-action
that is \'{e}tale-locally isomorphic to $\mathcal{G}$.  We denote the moduli stack of $\mathcal{G}$-bundles by
$\text{Bun}_{\mathcal{G}}$.
\end{defin}

We will also use the \emph{loop groups} associated to the parahoric group schemes.  For a $k$-algebra $R$,
let $R[[z]]$ and $R((z))$ denote the ring of formal power series and formal Laurent series with coefficients in $R$,
respectively.  Note that $R((z))$ is a $K$-algebra, and $R[[z]]$ is an $A$-algebra.
Then the loop groups associated to $G(K)$ and $\mathcal{P}$ are defined as follows.
\begin{defin}
For $K = k((z))$, the \emph{loop group} $LG$ associated to $G(K)$ is defined as the ind-scheme given by the functor
\[
R \mapsto G(R((z))).
\]
for any $k$-algebra $R$.  
The loop group $L^+\mathcal{P}$ associated to the parahoric subgroup $\mathcal{P}$ is the (infinite
dimensional) affine group scheme associated to the functor
\[
R \mapsto \mathcal{G}(R[[z]]),
\]
where $\mathcal{G}$ is the group scheme associated to $\mathcal{P}$.
\end{defin}

Loop groups and their affine flag varieties
are studied in much greater generality by Pappas and Rapoport in \cite{PR08}.

\subsubsection{Parahoric bundles as quotients of equivariant bundles}

Fixing weight data $\vec{w}$ we can understand parahoric bundles as quotients of bundles on a Galois cover $p:Y \rightarrow X$
that are equivariant with respect to the action of the Galois group $\Gamma$.

Let $E$ be a $\Gamma$-equivariant principal $G$-bundle over a Galois cover $p: Y \rightarrow X$, with a right $G$-action and 
left $\Gamma$-action.  If $y \in \mathcal{R}$ is a ramification point of $p$ then by the work done in
\cite{TWPB} we can find a formal neighborhood $N_y$ of $y$ such that $E$ is isomorphic over $N_y$
to the trivial bundle $N_y \times G$, with the action of $\Gamma_y$ given by $\gamma \cdot (\omega, g) = (\gamma \omega , \tau(\gamma) g)$,
where $\tau: \Gamma_y \rightarrow G$ does not depend on the formal parameter $\omega$.  We say that the \emph{local type}
of $E$ at $y$ is the conjugacy class of $\tau$.  The local type does not depend on the trivialization and is the same for
every ramification point over $p_i$.  The local type of $E$ is the collection of local types 
$\boldsymbol{\tau} = (\tau_1, \ldots, \tau_n)$ over each $p_i$.

\begin{defin}
The stack of $(\Gamma, G)$-bundles of local type $\boldsymbol{\tau}$ will be denoted $\equivBun$.  This
stack is a smooth and connected Artin stack.
\end{defin}

The weight data $\vec{w}$ determines ramification indices and local type representations.  
Let $m_i$ be a positive integer such that $\frac{m_i }{\ell}\cdot \lambda_i$ is integral.  Then it is well known 
that if $n \geq 3$ or $g \geq 2$ there exists a Galois covering $p: Y \rightarrow X$, ramified over each $p_i$ 
with ramification index $m_i$.  Let $\Gamma$ be the Galois group of $Y$ over $X$.
To each weight $\lambda_i$ we can associate a coweight $\frac{m_i }{\ell}\cdot \kappa(\lambda_i)$ and the associated cocharacter 
$\chi_i: \mathbb{G}_m \rightarrow T \subseteq G$. Let $\zeta_i$ be a primitive $m_i$-th root of unity, and 
let $\tau_i: \mathbb{Z}_{m_i} \rightarrow T$ be defined as $\tau_i(\gamma) = \chi_i(\zeta_i^\gamma)$.  
If $y \in Y$ is a ramification point, and $\Gamma_y$ the isotropy subgroup of $\Gamma$ at $y$, then 
$\Gamma_y \cong \mathbb{Z}_{m_i}$ and we can therefore think of $\tau_i$ as a representation of $\Gamma_y$.

Let $\mathcal{G}$ be the parahoric group scheme associated to $\vec{w}$ as above.  Then one of the main theorems in
\cite{BAS} is the following.

\begin{thm}{\cite{BAS}}
Given weight data $\vec{w} = (\lambda_1, \ldots, \lambda_n, \ell)$, we have a natural isomorphism of stacks
\[
\equivBun \xrightarrow{\sim} \normtext{Bun}_{\mathcal{G}}.
\]
where the ramification indices of $Y \rightarrow X$, local type $\boldsymbol{\tau}$, and $\mathcal{G}$ are
determined by $\vec{w}$ as above.
\end{thm}

A quick sketch of the proof will be useful for what follows.  Balaji and Seshadri identify $\equivBun$
with a stack of torsors over a group scheme $\mathcal{G}_0'$ over $Y$ equivariant with respect to the Galois action.
They then show that this stack is isomorphic to the stack of torsors over the invariant pushforward 
group scheme $\mathcal{G}' = p_*^{\Gamma} \mathcal{G}'_0$ over $X$.  Finally they identify $\mathcal{G}'$
with the parahoric group scheme $\mathcal{G}$.  

This theorem allows one to define the semistability of parahoric bundles in terms of semistability of equivariant
bundles.  It is easy to see that this definition does not depend on the choice of Galois cover $Y$.

\begin{defin}
We say that a $(\Gamma, G)$-bundle is \emph{$\Gamma$-semistable} if for every maximal parabolic $P \subseteq G$
and every $\Gamma$-equivariant $P$-reduction $\sigma: Y \rightarrow \mathcal{F}/P$ we have 
$\sigma^* \mathcal{F}(\mathfrak{g}/\mathfrak{p}) \geq 0$.  Stability is defined in the same way, replacing inequalities
with strict inequalities. 
A $\mathcal{G}$-bundle is said to be \emph{semistable} (resp. \emph{stable}) with respect to some weight data
$\vec{w}$ if the corresponding $(\Gamma, G)$-bundle is semistable (resp. stable).
\end{defin}

\subsection{Relative flag structures for parahoric bundles}
\label{flags}

In this section we study the morphism 
$\text{Bun}_{\mathcal{G}_\mathcal{Q}} \rightarrow \text{Bun}_\mathcal{G}$, where $\mathcal{G}_\mathcal{Q}$
is the group scheme associated to subgroups $\mathcal{Q}_i$ of the parahoric groups $\mathcal{P}_i$ defining $\mathcal{G}$.
The main results of this section 
are the construction of an isomorphism of $\text{Bun}_{\mathcal{G}_\mathcal{Q}}$ with the stack $\text{Parbun}_\mathcal{G}(\vec{\mathcal{Q}})$
 of $\mathcal{G}$-bundles with flags (Proposition \ref{FLGSTR}), and the identification of the fibers of 
$\text{Bun}_{\mathcal{G}_{\mathcal{Q}}} \rightarrow \text{Bun}_{\mathcal{G}}$ (Corollary \ref{ProjFib})
as connected flag varieties.
This can be seen as a generalization of the well-known facts that $\text{Bun}_{\mathcal{G}_{\mathcal{I}}} \cong \text{Parbun}_G$ --
here $\mathcal{G}_{\mathcal{I}}$ is the parahoric group scheme associated to the Iwahori subgroup $\mathcal{I}$ as
each point $p_1, \ldots, p_n$ --  and that the forgetful morphism $\text{Parbun}_G \rightarrow \text{Bun}_G$
has projective and connected geometric fibers.
We will use these results in section \ref{PProp} to show that conformal blocks descend to stacks of parahoric 
bundles.

\subsubsection{Special fibers of parahoric group schemes}\label{FIBERS}

Let $\mathcal{G}$ be a parahoric group scheme over $\spec{A}$ corresponding to a parahoric subgroup $\mathcal{P}$.
We want to describe the special fiber $\mathcal{G}(k)$.  Choose a rational cocharacter $\mu$ in the interior of the face of the fundamental
alcove corresponding to $\mathcal{P}$.  Let $m$ be an integer such that $m \mu$ is integral, and let 
$K' = k((z^{1/m})) = k((\omega))$ and $A' = k[[z^{1/m}]] = k[[\omega]]$.   Write $\Delta$ for the restriction of $m \mu$ to $\spec{K'}$,
and let $\tau$ be the representation of $\Gamma \cong \mathbb{Z}_m$ given by $\tau(\gamma) = m \mu(\zeta^\gamma)$,
where $\zeta$ is a primitive $m^{th}$ root of unity.  Then the isomorphism of $\mathcal{G}$ with the invariant pushforward
group scheme $\mathcal{G}'$ is the morphism induced by conjugation by $\Delta$.  In particular 
\[
\Delta \mathcal{P} \Delta^{-1} =\Delta \mathcal{G}(A) \Delta^{-1} = \mathcal{G}'(A) = G(A')^\Gamma,
\]
where the action of $\Gamma$ on $G(A')$ is the one induced by $\tau$: for $f \in G(A')$, 
$(\gamma f)(\omega) = \tau(\gamma) f(\gamma^{-1} \omega) \tau(\gamma)^{-1}$.  This identification induces an isomorphism of the
group schemes $\mathcal{G} \cong \mathcal{G}'$ because both group schemes are
\'{e}ttof\'{e}.  
	
Now considering $k$ as an $A$-module via the isomorphism $k \cong A/(z)$, we have $k \otimes_A A' \cong k[[\omega]]/(\omega^m)$. 
We will write this ring as $k[\epsilon]$. Then
$\mathcal{G}'(k) = G(k[\epsilon])^\Gamma$.  Furthermore the homomorphisms $k \rightarrow k[\epsilon] \rightarrow k$
induce homomorphisms $G \rightarrow G(k[\epsilon]) \rightarrow G$ that commutes with the action of $\Gamma$, which
is just conjugation by $\tau$ on $G$.  Then we have homomorphisms $C_G(\tau) \rightarrow \mathcal{G}'(k) \rightarrow C_G(\tau)$,
with the composition being the identity.  Then we have the following description of the special fiber of $\mathcal{G}$.

\begin{prop} \label{CLSDFIBER}
Fixing $\mu$, there is a canonical pair of homomorphisms $C_G(\tau) \xrightarrow{\iota} \mathcal{G}(k) \xrightarrow{\pi} C_G(\tau)$
with $\pi \circ \iota = Id$.  Furthermore, the kernel of $\pi$ is isomorphic to the group of $m^{th}$ order $\Gamma$-invariant deformations of 
the identity of $G$, and is the unipotent radical of $\mathcal{G}(k)$.  Finally, for any scheme $S$ the natural map
$L^+\mathcal{P}(S) \rightarrow \mathcal{G}(k)(S)$ is surjective.  Its composition with $\pi$ is given by conjugation by
$\Delta$ and setting $\omega$ equal to zero.  
\end{prop}
\begin{proof}
The kernel of $\pi$ contains the unipotent radical of $\mathcal{G}(k)$ since $C_G(\tau)$ is reductive.
 Let $f \in \text{ker}(\pi)$, and let $G \rightarrow \text{GL}(V)$ be any faithful representation.  Then clearly, 
we can identify $f$ with a unipotent element of $\text{GL}(V \otimes_k k[\epsilon])$, and therefore $\text{ker}(\pi)$
is unipotent.  Furthermore, shifting the parameter of $f$ by $a \in k$ gives an $m^{th}$-order deformation $f(a\omega)$
that is still $\Gamma$-invariant and in $\text{ker}(\pi)$.  Taking $a \rightarrow 0$ connects $f$ with the identity of $\mathcal{G}(k)$,
showing that $\text{ker}(\pi)$ is connected.  Therefore $\text{ker}(\pi)$ is the unipotent radical of $\mathcal{G}(k)$.

It remains to show that  
$L^+\mathcal{P}(S) \rightarrow \mathcal{G}(k)(S)$ is surjective for any scheme $S$.  This argument is essentially
identical to part of the proof of Lemma 2.5 in \cite{TWPB}.  

We use non-abelian cohomology, letting $\Gamma$ act on $G(S[[\omega]])$ and $G(S[\epsilon])$ as above, following the notation
and conventions in Serre's \emph{Cohomologie galoisienne} \cite{GALCOH}.  Let $G_n = G(S[[\omega]]/\omega^n)$.
We want to show that the natural map 
$\mathcal{P} \cong \text{H}^0(\Gamma, G(S[[\omega]])) \rightarrow \text{H}^0(\Gamma, G(S[\epsilon])) \cong \mathcal{G}(k)$ is surjective.
It is sufficient to show that each $\text{H}^0(\Gamma, G_{n+1}) \rightarrow \text{H}^0(\Gamma, G_n)$ is surjective,
since 
\[
\lim_{\leftarrow} \text{H}^0(\Gamma, G_n) \subseteq \text{H}^0(\Gamma, G(S[[\omega]])).
\]
Since $G$ is smooth, it follows that the morphism $G_{n+1} \rightarrow G_n$ is surjective for all $n$.
Then consider the short exact sequence 
\[
1 \rightarrow K_n \rightarrow G_{n+1} \rightarrow G_n \rightarrow 1,
\]
which induces an exact sequence of pointed sets
\[
\text{H}^0(\Gamma, G_{n+1}) \rightarrow \text{H}^0(\Gamma, G_n) \rightarrow \text{H}^1(\Gamma, K_n).
\]
It is easy to see then that the kernel $K_n$ is a $\mathbb{C}$-vector space: consider for example the case $n=1$,
where $K_n$ is isomorphic to the Lie algebra of $G$.  Therefore, since $\Gamma$ is finite, we see that $\text{H}^1(\Gamma, K_n)$
is trivial by \cite[Proposition 6]{HS53}. This proves the desired surjectivity.
\end{proof}

\subsubsection{The image of parahoric subgroups in $C_G(\tau)$}

We want to describe the image of parahoric subgroups $\mathcal{Q} \subseteq \mathcal{P}$ in $C_G(\tau)$.  
For a root $\alpha$, let $(\mu, \alpha)$ denote the pairing of characters and cocharacters, and square
brackets $[x]$ denote the smallest integer less than or equal to $x$.  Then the parahoric subgroup $\mathcal{P}$
associated to $\mu \in \mathcal{A}$ is defined as the group
\[
\mathcal{P} = \langle T(A), U_\alpha(z^{-[(\mu, \alpha)]} A), \alpha \in R \rangle,
\]
where $U_\alpha$ denotes the root group associated to $\alpha$ and $U_\alpha(z^{-[(\mu, \alpha)]} A)$
is the group $X_{\alpha + [(\mu, \alpha)]}$ fixing the affine half-apartment $A_{\alpha + [(\mu, \alpha)]}$.  
We have the following proposition.

\begin{lemma}
Let $\mathcal{Q} \subseteq \mathcal{P}$ be a parahoric subgroup corresponding to $\mu'$.  Then the
image of $\mathcal{Q}$ in $C_G(\tau)$ is exactly the group generated by $T$ and the root groups $U_\alpha$
such that $(\mu, \alpha) = [(\mu', \alpha)]$.
\end{lemma}
\begin{proof}
For any $\alpha \in R$, $\Delta U_\alpha(z^{-[(\mu', \alpha)]} A) \Delta^{-1} = U_\alpha(z^{(\mu, \alpha) - [(\mu', \alpha)]} A)$.
Therefore the result follows.
\end{proof}

\begin{cor}
$C_G(\tau)$ is the group generated by $T$ and the root groups $U_\alpha$ such that $(\mu, \alpha)$ is an integer,
and the image of the Iwahori subgroup $\mathcal{I}$ in $C_G(\tau)$ is the group generated
by $T$ and root groups $U_\alpha$ such that $(\mu, \alpha)$ is a nonpositive integer.
\end{cor}
\begin{proof}
For the description of $C_G(\tau)$, simply take $\mathcal{Q} = \mathcal{P}$ in the proposition above.

The Iwahori subgroup corresponds to a cocharacter $\mu'$ in the interior of the alcove, and therefore 
if $(\mu, \alpha) = [(\mu', \alpha)]$, then  $(\mu, \alpha)$ is either $0$ or $-1$.
\end{proof}

The images of sub-parahoric subgroups of $\mathcal{P}$ are in fact parabolic subgroups of $C_G(\tau)$. This
will allow us to identify the flags for parahoric bundles defined below with points in a (connected) flag variety.

\begin{prop}\label{CONNGRP}
The group $C_G(\tau)$ is a connected reductive subgroup of $G$, and the image of $\mathcal{I}$ in $C_G(\tau)$ is a Borel subgroup $B'$.
Furthermore $B'$ is the intersection of a Borel subgroup $B_\mu$ of $G$ with $C_G(\tau)$, and $B_\mu = w_\mu B w_\mu$, where
$w_\mu \in W$ is of order 2.  In particular, if $\mathcal{P}$ is maximal, then $C_G(\tau)$ is a subgroup
of $G$ of maximal rank.
\end{prop}
\begin{proof}
A proof that the centralizer of an element of a simply connected group is connected  can be found in the lecture notes on 
conjugacy classes by Springer and Steinberg \cite{SSCC}.
In the case that $\mathcal{P}$ is maximal, then $C_G(\tau)$ is a subgroup of maximal rank, as studied by Borel
and de Siebenthal \cite{BDS}.  
The Dynkin diagram of $C_G(\tau)$ is given by removing the vertices of the affine Dynkin diagram associated to $\mathcal{P}$.  

Let $B_\mu$ be the subgroup of $G$ generated by $T$ and $U_\alpha$ such that either $(\mu, \alpha) < 0$ or $\alpha \in R^+$ and 
$(\mu, \alpha) = 0$. Then we have that the image of $\mathcal{I}$ in $C_G(\tau)$ is 
just $C_G(\tau) \cap B_\mu$.  Let $P$ be the parabolic subgroup associated to $\mathcal{P}$, and let $w_\mu$ be 
the product of the longest words in $W$ and $W_P$.  This element $w_\mu$ switches positive and negative roots
for any root $\alpha$ such that $(\mu, \alpha) \neq 0$, and fixes all other roots.  Clearly then $w_\mu$ is order 2, and
$B_\mu = w_\mu B w_\mu$.
\end{proof}

Recall that the parahoric subgroups contained in $G(A)$ can also be defined as inverse images of parabolic subgroups of $G$.
The following proposition generalizes this fact to other parahorics subgroups.  

\begin{prop}\label{INVSIMG}
If $Q$ is the image of $\mathcal{Q}$ in $C_G(\tau)$, then the inverse image of $Q$ in $\mathcal{P}$ is
exactly $\mathcal{Q}$.
\end{prop}
\begin{proof}
The inverse image of $Q$ in $G(A')^\Gamma$ can be described as the group
\[
\langle T(A), U_\alpha(\omega^{k_\alpha} A), \alpha \in R \rangle
\]
for some non-negative integers $k_\alpha$.  There are two cases: either $(\mu, \alpha) = [(\mu', \alpha)]$ and $k_\alpha = 0$, or
$k_\alpha > 0$.  Note that $\alpha$ satisfies $(\mu, \alpha) = [(\mu', \alpha)]$ if and only if $-\alpha$ does.  Then in the first
case, we see that
\[
\Delta^{-1} U_\alpha(\omega^{k_\alpha} A) \Delta = U_\alpha(\omega^{-(\mu, \alpha)} A) = U_\alpha(\omega^{-[(\mu', \alpha)]} A).
\]
Now in the second case, we know that $(\mu, \alpha) > [(\mu', \alpha)]$.  In the case that $\alpha \in R^+$, we 
therefore know that $[(\mu', \alpha)] = 0$, since $0 \leq (\mu', \alpha), (\mu, \alpha) \leq 1$.  This implies that
$0 \leq (\mu', \alpha) < 1$, and $0 < (\mu, \alpha) \leq 1$.  We also know that $[(\mu', -\alpha)] < (\mu, -\alpha) < 0$,
so that in fact $0 < (\mu',\alpha), (\mu, \alpha) < 1$.  Therefore in this case we have
\[
\Delta^{-1} U_\alpha(\omega^{k_\alpha} A) \Delta = U_\alpha(\omega^{-[(\mu, \alpha)]} A) = U_\alpha(\omega^{-[(\mu', \alpha)]} A),
\]
finishing the proof.
\end{proof}

\subsubsection{Relative flag structures}

Now we return to the global situation over $X$, and define relative flag structures for parahoric bundles.

Let $\mathcal{G}$ be the group scheme over $X$ corresponding to parahoric subgroups $\mathcal{P}_1, \ldots, \mathcal{P}_n$, 
and let $\mathcal{Q}_i \subseteq \mathcal{P}_i$ be subgroups.  As above, let $\mu_i$ be the cocharacter 
associated to each $\mathcal{P}_i$, with associated $\Delta_i \in G(K')$, and representation $\tau_i: \Gamma_{p_i} \rightarrow T$.
Let $Q_i$ be the image of each $\mathcal{Q}_i$ in $C_G(\tau_i)$, and let $\widetilde{Q}_i$ be the inverse image
of $Q_i$ in the closed fiber $\mathcal{G}(p_i)$.  Then we have the following definition.

\begin{defin}
Let $\text{Parbun}_\mathcal{G}(\vec{\mathcal{Q}})$ be the moduli stack of $\mathcal{G}$-bundles together
with flags $\overline{g}_i \in \mathcal{E}(p_i)/\widetilde{Q}_i \cong C_G(\tau_i)/Q_i$.
\end{defin}

Then we claim the following.

\begin{prop}\label{FLGSTR}
We have an isomorphism of stacks:
\[
\normtext{Bun}_{\mathcal{G}_\mathcal{Q}} \cong \normtext{Parbun}_\mathcal{G}(\vec{\mathcal{Q}}).
\]
\end{prop}
\begin{proof}
For simplicity, we describe the morphisms pointwise;  the same constructions work with families.

Then there is a natural morphism $\mathcal{G}_\mathcal{Q} \rightarrow \mathcal{G}$ which induces a morphism of stacks 
$\text{Bun}_{\mathcal{G}_\mathcal{Q}} \rightarrow \text{Bun}_\mathcal{G}$. 
If $\mathcal{E}_\mathcal{Q}$ is a $\mathcal{G}_\mathcal{Q}$-bundle, via this projection we get a $\mathcal{G}$-bundle
$\mathcal{E}$.  We also have a natural morphism $\mathcal{E}_\mathcal{Q} \rightarrow \mathcal{E}$ and
its restriction to each $p_i$: $\mathcal{E}_\mathcal{Q}(p_i) \rightarrow \mathcal{E}(p_i)$.  This morphism
gives a canonical point in $\mathcal{E}(p_i)/\widetilde{Q}_i$, since $\mathcal{E}_\mathcal{Q}(p_i)/\widetilde{Q}_i$
is simply a point.  This defines a morphism 
$\normtext{Bun}_{\mathcal{G}_\mathcal{Q}} \rightarrow \text{Parbun}_\mathcal{G}(\vec{\mathcal{Q}})$. 

We can also define this morphism in terms of a trivialization.  
Given a $\mathcal{G}_\mathcal{Q}$-bundle, we can describe it in terms of transition functions as
\begin{align*}
(\mathcal{E}_0)_{|U_{p_i}^*} &\xrightarrow{\sim} (\mathcal{E}_i)_{|U_{p_i}^*} \\
(z, g) &\mapsto  (z,\Theta_i(z)g),
\end{align*}
where $U_{p_i}^* \cong \text{Spec}(K)$ is a formal neighborhood around $p_i$.
Then the associated $\mathcal{G}$ bundle is given by the same transition functions, and we take the flag
$\overline{e} \in C_G(\tau_i)/Q_i \cong \mathcal{E}(p_i)/\widetilde{Q}_i$, with respect to this trivialization.
Clearly if we change the trivialization, we get the same $\mathcal{G}$-bundle and flag, since the new
transition function $\Theta_i'$ is given by $\Theta_i' = f\Theta_ig$, where $f \in \mathcal{Q}_i$, and $g$
is the restriction of a morphism $g: X^* \rightarrow G$ to $U_i^*$.  Multiplication on the right by $g$
can be accounted for by changing the trivialization of $\mathcal{E}$ over $X^*$, while multiplication on the
left by $f$ fixes the flag $\overline{e}$.

Going the other direction, we simply choose a trivialization of the underlying $\mathcal{G}_\mathcal{M}$-bundle such that
the flag is $\overline{e}$, then use these transition functions to construct the $\mathcal{G}$-bundle.  We can always do
this because the morphism $L^+\mathcal{M} \rightarrow \mathcal{G}(k)$ is surjective by Prop \ref{CLSDFIBER}.  
By Prop \ref{INVSIMG}, the subgroup of $\mathcal{P}_i$ fixing the flag $\overline{e}$ is exactly the parahoric 
subgroup $\mathcal{Q}_i$, so this definition does not depend on the choice of trivialization.
\end{proof}

\begin{cor}
\label{ProjFib}
For any parahoric subgroups $\mathcal{Q}_i \subseteq \mathcal{P}_i$ the morphism 
$\normtext{Bun}_{\mathcal{G}_{\mathcal{Q}}} \rightarrow \normtext{Bun}_{\mathcal{G}}$ is a smooth, proper and
surjective representable morphism with connected and projective geometric fibers.
\end{cor}
\begin{proof}
By the above proposition, the morphism is clearly representable, and the geometric fibers are isomorphic to
the product
\[
\prod_{i=1}^n C_G(\tau_i)/Q_i.
\]
Then by Proposition \ref{CONNGRP}, the geometric fibers are connected and projective.
\end{proof}

\subsection{Line bundles on stacks of parahoric bundles}
\label{PProp}

In this section we use the above description of $\text{Bun}_{\mathcal{G}}$ to prove that conformal blocks descend to this stack,
finishing the proof of Theorem \ref{GlblSecDesc}.  The key to the proof is that the fibers of 
$\normtext{Parbun}_G \rightarrow \normtext{Bun}_{\mathcal{G}}$ are connected and projective.
First we show that the locus of semistable bundles can be defined like semistability is defined in GIT, which will
be needed later.

\begin{prop} \label{SEMIEQ}
An equivariant bundle $E \in \equivBun$ is $\Gamma$-semistable if and only if there exists $s \in \normtext{H}^0(\equivBun, D(V)^N)$
for some integer $N>0$ and a faithful representation $V$ of $G$ such that $s(E) \neq 0$.
\end{prop}
\begin{proof}
Suppose $E$ is $\Gamma$-semistable.  Then $E$ is also a semistable $G$-bundle.  This follows from the uniqueness of the canonical
reduction of an unstable $G$-bundle (see section 2.4 in \cite{TWPB}).   Now the bundle $E$ corresponds to a point 
$x \in \mathbb{M}$ of the moduli space $\mathbb{M}$ of $G$-bundles over $Y$.  Then since  a determinant
of cohomology line bundle descends to an ample bundle $\mathcal{L}$ over $\mathbb{M}$ \cite{SNR}, there is a section
$s \in \text{H}^0(\mathbb{M}, \mathcal{L}^N)$ for some $N>0$ such that $s(x) \neq 0$.  In \cite{LASOR} Laszlo and Sorger
showed that $\text{H}^0(\mathbb{M}, \mathcal{L}^N) = \text{H}^0(\text{Bun}_G, \mathcal{L}^N)$, so pulling back and extending $s$ over
$\equivBun$ gives a section $s \in \normtext{H}^0(\equivBun, D(V)^N)$ such that $s(E) \neq 0$.

Now suppose there is an $s \in \normtext{H}^0(\equivBun, D(V)^N)$ such that $s(E) \neq 0$.  For the sake of contradiction,
suppose that $E$ is not $\Gamma$-semistable.  Then there is a (unique) canonical $P$-reduction $\phi_E: Y \rightarrow E/P$
that is a maximum violator of semistability.  This $P$-reduction gives a one-parameter family of
equivariant bundles $f: \mathbb{A}^1 \rightarrow \equivBun$.  But by Mumford's numerical criterion for semistability,
since $s(E) \neq 0$, the index $\mu(E, f)$ is non-negative, which contradicts the assumption that $\phi$ is a maximal violator
of semistability for $E$, since the index is a positive multiple of the degree of $\phi_E$.  (For more details on the construction of
$f$ and calculation of its index, see the proof of our Proposition \ref{PULLBCK} and Lemma 3.16 in \cite{BK}.)   
Therefore $E$ is $\Gamma$-semistable. 
\end{proof}

Note that the same result holds for $\text{Bun}_{\mathcal{G}}$ and $\mathcal{L}_{\vec{w}}$, assuming this line bundle
descends, which we prove below.
The following proposition contains the basic geometric argument behind the proof of Theorem \ref{GlblSecDesc},
assuming we know that the line bundle itself descends.

\begin{prop}
\label{ZarIsom}
Suppose $f: \mathcal{X} \rightarrow \mathcal{Y}$ is a representable morphism of Artin stacks, where $\mathcal{Y}$ is smooth
over $k$ and 
$f$ is smooth, proper and surjective with connected geometric fibers.  Then 
$f_*(\mathcal{O}_\mathcal{X}) = \mathcal{O}_\mathcal{Y}$,
and for any line bundle $\mathcal{L}$ over $\mathcal{Y}$, the pullback via $f$ induces an isomorphism of global
sections: $\normtext{H}^0(\mathcal{Y}, \mathcal{L}) \xrightarrow{\sim} \normtext{H}^0(\mathcal{X}, \mathcal{L})$.
\end{prop}
\begin{proof}
By Stein factorization of Artin stacks (see \cite{ARTINSHV}) $f$ factors as 
$\mathcal{X} \xrightarrow{f'} \mathcal{Y}' \xrightarrow{e} \mathcal{Y}$, where
$f'$ is proper with connected fibers, $f'_*(\mathcal{O}_V) \cong \mathcal{O}_{U'}$, and $e$ is finite.
But since $f$ is surjective and has connected fibers, $e$ must have connected fibers.  But a finite morphism
with connected fibers is an isomorphism since we are working over an algebraically closed field of characteristic $0$,
and $\mathcal{Y}$ is normal.  Therefore $f_*(\mathcal{O}_\mathcal{X}) = \mathcal{O}_\mathcal{Y}$,
and by the projection formula, $\normtext{H}^0(\mathcal{Y}, \mathcal{L}) \rightarrow \normtext{H}^0(\mathcal{X}, \mathcal{L})$
is an isomorphism.
\end{proof}

We can now finish the proof of Theorem \ref{GlblSecDesc}.

\begin{proof}[Proof of Theorem \ref{GlblSecDesc}]

By Proposition \ref{ZarIsom} the pullback of global sections of any line bundle on $\text{Bun}_{\mathcal{G}}$ to $\text{Parbun}_G$
is an isomorphism. It remains to show that $\mathcal{L}_{\vec{w}}$ descends to $\text{Bun}_\mathcal{G}$, assuming $\mathcal{G}$ is the parahoric
group scheme associated to $\vec{w}$.  First note that by the work in section 6 of \cite{BK}, a power of the line bundle $\mathcal{L}_{\vec{w}}$ can
be identified with the pullback to $\text{Parbun}_G$ of a determinant of cohomology bundle on $\equivBun$.  Therefore
by Balaji and Seshadri's identification of the stacks of parahoric bundles and equivariant bundles, a power of $\mathcal{L}_{\vec{w}}$ 
descends to $\text{Bun}_\mathcal{G}$.  In particular, a power of $\mathcal{L}_{\vec{w}}$ is trivial over the fibers of 
$f: \text{Parbun}_G \rightarrow \text{Bun}_\mathcal{G}$.  Now since the fibers of $f$
are isomorphic to a product of connected flag varieties, the Picard groups of the fibers are torsion free, and therefore
$\mathcal{L}_{\vec{w}}$ itself is trivial over the fibers of $f$.

We want to show that $f_*(\mathcal{L}_{\vec{w}})$ is a line bundle, and that its pullback to $\text{Parbun}_G$ is $\mathcal{L}_{\vec{w}}$.
(The following argument is essentially a solution of exercise III.12.4 in \cite{HART}.)  Let $U \rightarrow \text{Bun}_\mathcal{G}$ be a smooth morphism
and let $V$ be the fiber product of $U$ and $\text{Parbun}_G$.  By definition, the pullback of $f_*(\mathcal{L}_{\vec{w}})$
to $U$ is the pushforward of the pullback of $\mathcal{L}_{\vec{w}}$ to $V$.  Since $\mathcal{L}_{\vec{w}}$ is trivial on the fibers,
and the fibers are projective and connected, we have $\text{H}^0(V_y, \mathcal{L}_{\vec{w}}) = k$ for any $y \in U$.
Therefore by Grauert's Theorem, $f_*(\mathcal{L}_{\vec{w}})$ is locally free of rank $1$ over $U$, and therefore over 
$\text{Bun}_\mathcal{G}$ \cite[Corollary III.12.9]{HART}.  Now by the adjoint property of pullbacks, there is a natural
morphism of sheaves $f^*f_*(\mathcal{L}_{\vec{w}}) \rightarrow \mathcal{L}_{\vec{w}}$.  To show this is an isomorphism,
it is sufficient to check it on fibers.  Let $x \in \text{Parbun}_G$ be a $k$-valued point, and $y$ its image in $\text{Bun}_\mathcal{G}$.
Then the fiber of $f^*f_*(\mathcal{L}_{\vec{w}})$ over $x$ is $\text{H}^0((\text{Parbun}_G)_y, \mathcal{L}_{\vec{w}})$, and the morphism
to the fiber of $\mathcal{L}_{\vec{w}}$ is simply the evaluation map.  But since $\text{H}^0((\text{Parbun}_G)_y, \mathcal{L}_{\vec{w}})$
is the space of constant functions on $(\text{Parbun}_G)_y$, this map is nonzero, and therefore 
$f^*f_*(\mathcal{L}_{\vec{w}}) \cong \mathcal{L}_{\vec{w}}$.
\end{proof}

\begin{cor}
Let $X \cong \mathbb{P}^1$.  Then for any weight data $\vec{w}$, the space  $\normtext{H}^0(\normtext{Bun}_{\mathcal{G}}, \mathcal{L}_{\vec{w}}^N)$ is nonzero for some $N$ if and
only if $\vec{w}$ is in the multiplicative polytope.
\end{cor}
\begin{proof}
This follows from Theorem 5.2 in \cite{BK} and Theorem \ref{GlblSecDesc}.
\end{proof}

\section{Beginning of the proof of the reduction theorem: Stacks of $P$-reductions}\label{PReductionSec}

We are now ready to begin the proof of the reduction theorem for conformal blocks.  First we want to outline
the strategy of the proof in the language of parahoric bundles and equivariant bundles.
For the remainder of the proof of the reduction theorem we will fix the following data. Assume $X \cong \mathbb{P}^1$
and fix distinct points $p_1, \ldots, p_n \in X$.  Let $\vec{w}$ be weight 
data in the multiplicative polytope.  Assume that $\vec{w}$ lies on a face of the polytope corresponding to 
the quantum product $\sigma_{u_1} \ast \cdots \ast \sigma_{u_n} = q^d \text{[pt]}$ in $\text{QH}^*(G/P)$.
Let $\mathcal{G}$ be the parahoric group scheme over $X$ corresponding to $\vec{w}$, and let
$\equivBun$  be a stack of equivariant bundles over a curve $Y$ such that $\equivBun \cong \text{Bun}_{\mathcal{G}}$.
Let $L \subseteq P$ be the Levi subgroup containing $T$, and $L' = [L,L]$.

Consider the following diagram.
\[
\includegraphics{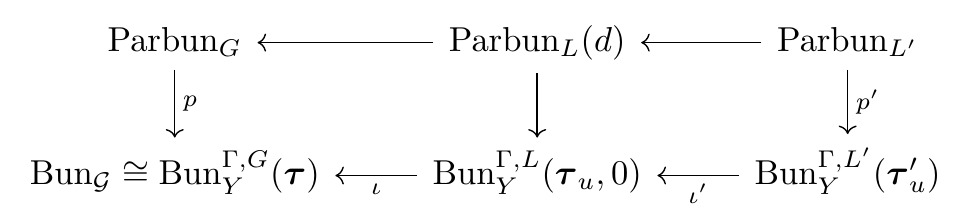}
\]
The morphisms $p$ and $p'$ are the projection morphisms for parahoric bundles discussed above.  The morphism $\iota$
is simply induced by extension of structure group.  The morphism $\iota'$ is induced by extension of structure group
and some additional non-canonical twisting if $d>0$, and will be described in more detail in section \ref{ProofSec}.
The other morphisms are the natural ones making the diagram commutative.

The basic strategy of the proof is to use these morphisms to prove that there is a natural isomorphism of global
sections of the line bundle $\mathcal{L}_{\vec{w}}$ over $\text{Parbun}_G$ and the associated line bundle $\mathcal{L}_{\vec{w}'}$
over $\text{Parbun}_{L'}$.  By the results in section \ref{ParahoricSec}, both $p$ and $p'$ induce an isomorphism of 
global sections.  The proof that $\iota'$ induces an isomorphism of global sections is similar
to the argument in section 7 of \cite{BK}, and we prove it in section \ref{ProofSec}.

In order to show that $\iota$ induces an isomorphism of sections, we use a method originally due to Ressayre \cite{RSS10}.
We start with a stack $\mathcal{C} \rightarrow \text{Parbun}_G$, the fibers of which correspond to $P$-reductions
of parabolic bundles of degree $d$ and relative position $(u_1, \ldots, u_n)$.  Our cohomology assumption guarantees
that this morphism is \emph{birational}, by which we simply mean there is an open subset of $\mathcal{C}$
mapping isomorphically to its image in $\text{Parbun}_G$.  We then embed this stack into a larger stack $\mathcal{Y}$, containing
$\mathcal{C}$ as a dense substack.  This stack fits into the following diagram.
\[
\includegraphics{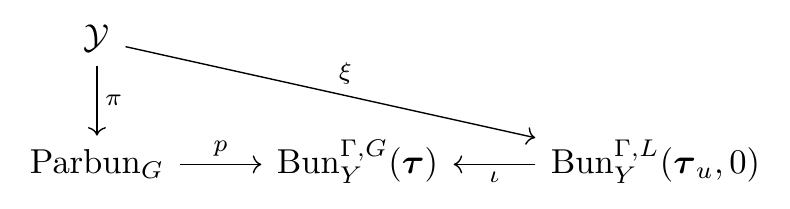}
\]
This diagram is \emph{not} $2$-commutative.  However, it does induce a commutative diagram of global sections via pullback 
(see Proposition \ref{PULLBCK}).
The main theorem of this section is that $\pi$ is proper over the semistable locus of $\text{Parbun}_G$.

\begin{thm}
\label{PropThm}
The reduction stack $\mathcal{C}$ is embedded in a stack $\mathcal{Y}$ over $\text{Parbun}_G$ such that
the restriction to semistable bundles $\pi: \mathcal{Y}^{ss} \rightarrow \text{Parbun}_G^{ss}$ is proper.
\end{thm}

The proof then goes as follows.  By a version of Zariski's main theorem, pullback via the birational, proper morphism 
$\pi: \mathcal{Y}^{ss} \rightarrow \text{Parbun}_G^{ss}$ induces an isomorphism of global sections for any line bundle.  
We will show that $\xi$ is surjective (see Proposition \ref{PULLBCK} and Theorem \ref{PLIFT}) so pullback via $\xi$ is injective.  
Then by a simple diagram chase, pullback via $\iota$ is an isomorphism.

\subsection{Universal reduction stacks}

We begin the construction of $\mathcal{Y}$ by reviewing  the analogous space in the 
``classical" case, replacing conformal blocks with spaces of invariants, and parabolic bundles with tuples of flags.
In \cite{BKR} Belkale, Kumar and Ressayre use these varieties to prove a generalization of Fulton's conjecture. 

Let $u_1, \ldots ,u_n \in W^P$.  Then $B$ fixes the Schubert cell $C_{u_i}$, and we can form the fiber bundle 
$\mathfrak{C}_{u_i} = G \times^B C_{u_i}$ over $G/B$.  Note that we have a natural projection map 
$\mathfrak{C}_{u_i} \rightarrow G/B$ and a map $\mathfrak{C}_{u_i} \rightarrow G/P$ defined by 
$[g,x] \mapsto gx$.  Let $\mathfrak{C} = \mathfrak{C}_{u_1} \times \cdots \times \mathfrak{C}_{u_n}$.  Then the \emph{universal intersection scheme} $\mathcal{C}$
is defined as the fiber product of the map $\mathcal{C} \rightarrow (G/P)^n$ with the diagonal $\delta : G/P \rightarrow (G/P)^n$.
So we have natural maps $g : \mathcal{C} \rightarrow (G/B)^n$
and $h : \mathcal{C} \rightarrow G/P$, and the points of $\mathcal{C}$
correspond to tuples $(\overline{g}_1, \ldots, \overline{g}_n, x)$ with each $\overline{g}_i \in G/B$, and $x \in G/P$
in the intersection of $g_1 C_{u_1}, \ldots, g_n C_{u_n}$. In a similar way we define the 
universal intersection schemes for the closed Schubert variety $X_{u_i}$ and its smooth locus $Z_{u_i}$, denoting them 
$\mathcal{X}$ and $\mathcal{Z}$ respectively.

The generalization to parabolic bundles is the reduction stack $\mathcal{C} \rightarrow \text{Parbun}_G$,
which replaces points in $G/P$ with $P$-reductions of specified degree and relative Schubert condition.
This stack was used by Belkale and Kumar in their work determining the irredundant inequalities of the
multiplicative polytope \cite{BK}.

\begin{defin}
The \emph{universal reduction stack $\mathcal{C}$ relative to $u_1, \ldots, u_n$ and degree $d$} is the stack of pairs of parabolic bundles 
$\widetilde{\mathcal{E}} \in \text{Parbun}_G$ and $P$-reductions $\mathcal{P} \times^P G \cong \mathcal{E}$ of degree $d$ and 
relative Schubert position $u_1, \ldots, u_n$ (see section \ref{Moduli}).  The \emph{smooth reduction stack} $\mathcal{Z}$ 
also includes $P$-reductions with relative positions $w_1, \ldots, w_n$ such that $C_{w_i} \subseteq Z_{u_i}$,
where $Z_{u_i}$ is the smooth locus of the closed Schubert variety $X_{u_i}$.
\end{defin}

More formally, we can describe $\mathcal{C}$ as a locally closed substack of $\text{Parbun}_{G,P}$, the moduli
stack of parabolic bundles paired with $P$-reductions.  Let $\widetilde{\mathcal{E}}$ be a family of parabolic bundles
over $S$, and 
write $\mathcal{E}_i$ for the restriction of $\mathcal{E}$ to $\{p_i\} \times S$.
Now we can consider $\mathfrak{C}_{u_i}$ as a locally closed subvariety of 
$G/B \times G/P$ via the morphism $\mathfrak{C}_{u_i} \rightarrow G/B \times G/P$ sending $(g, x) \mapsto (\overline{g}, gx)$.  
In a similar way
we can consider the fiber bundle $\mathcal{E}_i \times^B \mathfrak{C}_{u_i}$ as a locally closed subscheme of 
$\mathcal{E}_i/B \times \mathcal{E}_i/P$.  Now the $B$- and $P$-reductions of $\mathcal{E}_i$ correspond to a section
$s_i: S \rightarrow \mathcal{E}_i/B \times \mathcal{E}_i/P$.  The Schubert condition is that $s_i$ must factor through 
$\mathcal{E}_i \times^B \mathfrak{C}_{u_i}$ for each $i$.  

Clearly then $\mathcal{C}$ is a locally closed substack of $\text{Parbun}_{G,P}$.  Furthermore, since the morphism 
$\text{Parbun}_{G,P} \rightarrow \text{Parbun}_G$ is representable (see Proposition \ref{PREDREP}), so is $\mathcal{C} \rightarrow \text{Parbun}_G$,
and therefore $\mathcal{C}$ is an algebraic stack.
Similarly we have stacks $\mathcal{Z}$ and $\mathcal{X}$ which are representable over $\text{Parbun}_G$.  We
also have natural open embeddings $\mathcal{C} \subseteq \mathcal{Z} \subseteq \mathcal{X}$.
  
Our next task will be to prove some properties of these intersection stacks.

\subsubsection{Basic geometric properties of $\mathcal{Z}$}

It will be important later that $\mathcal{C}$ is reduced and irreducible.  To prove this, we 
show that $\mathcal{Z}$ is smooth over $\text{Spec}(k)$ and irreducible.  The same properties about
$\mathcal{C}$ follow since $\mathcal{C}$ is an open substack of $\mathcal{Z}$.

\begin{prop}
The stack $\mathcal{Z}$ is smooth over $\text{Spec}(k)$ and irreducible.
\end{prop}
\begin{proof}
Consider the natural projection $\mathcal{Z} \rightarrow \text{Bun}_{G,P}(d)$, where $\text{Bun}_{G,P}(d)$ is the
stack of principal $G$-bundles over $X$ paired with degree $d$ $P$-reductions.
Since $\text{Bun}_{G,P}(d) \cong \text{Bun}_{P}(d)$ is smooth over $\text{Spec}(k)$, 
it is sufficient to show that $\mathcal{Z} \rightarrow \text{Bun}_{G,P}(d)$ is smooth.  
It is easy to see that the fibers of the projection $\mathcal{Z} \rightarrow \text{Bun}_{G,P}$ are locally closed 
subschemes of Hilbert schemes, and therefore locally of finite type.  Finally,
$\mathcal{Z} \rightarrow \text{Bun}_{G,P}(d)$ is formally smooth because each
projection $\mathcal{E}_i \times^B \mathfrak{Z}_{u_i} \rightarrow \mathcal{E}_i/P$ is smooth (see \cite{BKR} section 5).
Therefore $\mathcal{Z}$ is smooth.

In particular the induced morphism of topological spaces $|\mathcal{Z}| \rightarrow |\text{Bun}_{G,P}(d)|$
is open. 
Furthermore $\text{Bun}_P(d)$ is
irreducible, since $\pi_0(\text{Bun}_P) \cong \pi_1(P) \cong \mathbb{Z}$ (cf. \cite{HFF} for a proof).
We claim that the fibers of $\mathcal{Z} \rightarrow \text{Bun}_{G,P}(d)$ over $k$-rational points are also irreducible. 
After choosing local trivializations of the $G$-bundle, the fiber of the projection over the point 
$\text{Spec}(k) \rightarrow \text{Bun}_{G,P}(d)$ is simply the independent
choices of flags in $G/B$ satisfying the Schubert conditions with respect to each point $x_i \in G/P$ given by
the $P$-reduction $\mathcal{P}$ and our chosen Weyl group elements $u_i$.  It is easy to see that each space of
possible choices in $G/B$ is an irreducible variety.  Therefore the fiber over the given point is a product
of irreducible varieties, which is irreducible.

Now we show that $\mathcal{Z}$ is irreducible.  Let $U_1$, $U_2$ be nonempty open subsets of $|\mathcal{Z}|$.  
Then since the projection 
$|\mathcal{Z}| \rightarrow |\text{Bun}_{G,P}(d)|$ is open, and the stack $\text{Bun}_{G,P}(d)$ is 
irreducible, the images of $U_1$ and $U_2$ in $|\text{Bun}_{G,P}(d)|$ are open and intersect non-trivially; say the 
intersection is $V$, a nonempty open subset of $|\text{Bun}_{G,P}(d)|$.  Then 
we have that the fiber over a $k$-rational point of $V$ intersects both $U_1$ and $U_2$
non-trivially.  But we know that such a fiber is irreducible, and therefore that
$U_1$ and $U_2$ must intersect non-trivially.
\end{proof}

\subsection{Lifting families of $P$-reductions}

We need to embed $\mathcal{C}$ in a larger stack $\mathcal{Y}_0$ in order to construct $\mathcal{Y}$
by taking the closure.  Our approach is to first lift the $P$-reductions in $\mathcal{C}$ to
$P$-reductions of equivariant bundles.

Say $\widetilde{\mathcal{E}}$ is a parabolic $G$-bundle on $X$, and $E$ is a $(\Gamma,G)$-bundle on $Y$, of local type
$\boldsymbol{\tau}$ such that $E$ is the image of $\widetilde{\mathcal{E}}$ in $\equivBun$.  Then
$P$-reductions of $\mathcal{E}$ can clearly be lifted individually to $E$, since generically $E$ is just the pullback
of $\mathcal{E}$ to $Y$ with the trivial $\Gamma$ structure (see below), and the closure of a generic $P$-reduction exists and is
unique.  However in families this process is \emph{discontinuous}.  This can be seen by considering a connected family 
of $P$-reductions of $\mathcal{E}$ that jumps in parabolic degree: once lifted to $E$ this becomes a change in plain
degree.

Let $\equivBunPTwo{0}$ be the
stack of $(\Gamma, G)$-bundles of local type $\boldsymbol{\tau}$ together with invariant $P$-reductions of local
type $\boldsymbol{\tau}_u = (u_1^{-1}\tau_1u_1, \ldots, u_n^{-1}\tau_nu_n)$ and degree $0$.
The goal of this section is to prove the following theorem.

\begin{thm}\label{PLIFT}
There exists a representable, surjective morphism $\mathcal{C} \rightarrow \equivBunPTwo{0}$.
\end{thm}

\subsubsection{$\equivBun \xrightarrow{\sim} \normtext{Bun}_{\mathcal{G}}$ in terms of transition functions}

We need to make the identification of parahoric bundles with equivariant bundles more explicit for what follows.  
For more details see \cite{BAS}.

By \cite{DESC} and \cite{TWPB} we can describe a $(\Gamma, G)$-bundle $E$ as follows.  Let $\mathcal{R}$ be the
ramification locus of $p: Y \rightarrow X$.  Let $E_0$ be the trivial $G$-bundle over 
$Y^* = Y \setminus \mathcal{R}$ with trivial $\Gamma$ action, and for each $y \in \mathcal{R}$ such that $p(y) = p_i$, let $E_y$ be 
the trivial $(\Gamma, G)$-bundle over $N_y$ with local type $\tau_i$.  Then $E$ is isomorphic to the $(\Gamma, G)$-bundle
corresponding to a choice of $\Theta_i \in G(K)$, giving transition functions
\begin{align*}
(E_0)_{|N_y^*} &\xrightarrow{\sim} (E_y)_{|N_y^*} \\
(\omega, g)         &\mapsto  (\omega,\Delta_i(\omega) \Theta_i(\omega)g),
\end{align*}
where $\Delta_i \in G(K')$ is associated to $\mu_i$ as in section \ref{ParahoricSec}. Note that the choice of $\Theta_i$ is not unique.  Changing the trivialization of $E_0$ multiplies $\Theta_i$ on
the right by an element of $G(K)$, and changing the trivialization of $E_y$
multiplies $\Theta_i$ on the left by an element of the parahoric subgroup $\mathcal{P}_i$ corresponding to $\mu_i$.

Now $\equivBun \xrightarrow{\sim} \text{Bun}_{\mathcal{G}}$ is given as follows.
Let $F \in \equivBun$ be the bundle where each $\Theta_y = e$.    
Let $\mathcal{G}_F$ be the adjoint bundle $F \times^G G$, with $G$ acting on itself by 
conjugation, and let $\mathcal{G}' = p_*^\Gamma(\mathcal{G}_F)$ be the invariant push-forward of this group scheme.  
The group scheme $\mathcal{G}_F$ can be identified with the sheaf of automorphisms of $F$, and $\mathcal{G}'$ is a 
representable by a smooth group scheme over $X$ isomorphic to the parahoric group scheme $\mathcal{G}$.
Let $\text{Isom}(E,F)$ be the sheaf of local isomorphisms of $E$ and $F$.  This sheaf is a  right 
$\mathcal{G}_F$-torsor.  Then $p_*^\Gamma(\text{Isom}(E,F))$ is 
representable by a smooth variety over $X$ and is naturally a right $\mathcal{G}$-torsor.  

Let $E$ be an equivariant bundle with transition functions given by $\Delta_i(\omega) \Theta_i$
as above.
Then the $\mathcal{G}$-bundle $\mathcal{E}$ corresponding to  $E$ can be described as follows.  If $\mathcal{E}_0 = X^* \times G$ is the trivial $G$-bundle 
and $\mathcal{E}_i = \mathcal{G}_i$ is the parahoric group scheme corresponding to each $\mu_i$,
then $\mathcal{E}$ is isomorphic to the $\mathcal{G}$-bundle given by $\Theta_i \in G(A)$, giving transition morphisms
\begin{align*}
(\mathcal{E}_0)_{|U_{p_i}^*} &\xrightarrow{\sim} (\mathcal{E}_i)_{|U_{p_i}^*} \\
(z, g) &\mapsto  (z,\Theta_i(z)g).
\end{align*}

\subsubsection{Construction of $\mathcal{C} \rightarrow \equivBunPTwo{0}$}

Consider a morphism $S \rightarrow \mathcal{C}$, where $S$ is an arbitrary scheme.  This corresponds to a family of parabolic
bundles over $S$ and a family of $P$-reductions of the underlying family of $G$-bundles $\mathcal{E} \rightarrow X \times S$.  
Now by the uniformization theorem in \cite{UNIF} there is an \'{e}tale cover $\widetilde{S} \rightarrow S$ such that $\mathcal{E}$
is trivialized over $X^* \times \widetilde{S}$ and $U_x \times \widetilde{S}$ for each $x \in \{p_1, \ldots, p_n \}$.  Let $\mathcal{E}_0$
be the restriction of $\mathcal{E}$ to $X^* \times \widetilde{S}$, and let $\mathcal{E}_x$ be the restriction of $\mathcal{E}$ to
each $U_x \times \widetilde{S}$. 
Suppose $\Theta_x: U_x^* \times \widetilde{S} \rightarrow G$ gives the transition map over
$U_x^* \times \widetilde{S}$ with respect to some trivialization of $\mathcal{E}_0$
and $\mathcal{E}_x$.  
Let $\overline{g}_x: \widetilde{S} \rightarrow G/B$ be the family of flags at $x$.  Then taking a further refinement of $\widetilde{S}$
(which we continue to denote $\widetilde{S}$) we can lift this morphism to $g_x: \widetilde{S} \rightarrow G$.  So clearly we can choose
a trivialization of $\mathcal{E}$ near $x$ such that the family of flags is identically trivial.
Then these transition maps are also transition maps for the corresponding $\mathcal{G}$-bundle,
and by the above discussion the corresponding $\Gamma$-equivariant bundle $E$ is given by the transition
functions $\Delta_y \Theta_x$ for each $p(y) = x$.

Now locally near $x$, with respect to the above trivializations, the 
$P$-reduction of $\mathcal{E}$ corresponds to a morphism $\psi: U_x \times \widetilde{S} \rightarrow G/P$.  Then since the flags
are trivial, the generic $P$-reduction of the corresponding $\mathcal{G}$-bundle near $x$ is just the restriction of 
$\psi$ to $U^*_x \times \widetilde{S}$, and the generic $P$-reduction of $E$ is given by 
$\Delta_y \psi: N^*_y \times \widetilde{S} \rightarrow G/P$.  Our first task is to show that this morphism extends to
all of $N_y \times \widetilde{S}$, and that therefore the generic $P$-reduction of $E$ extends to all of $Y \times S$.

\begin{lemma}\label{NOPOLE}
For a scheme $S$ and a morphism $\psi: U \times S \rightarrow G$, if the restriction $\psi_0$ to $\normtext{Spec}(k) \times S$
factors through $B$, then $\Delta \psi \Delta^{-1}$ is defined on all of $N \times S$.
\end{lemma}
\begin{proof}
Since $\psi_0$ lands in $B$, the morphism $\psi$ corresponds to a morphism $S \rightarrow L^+\mathcal{I} \rightarrow L^+\mathcal{P}$,
where $\mathcal{P}$ is the parahoric subgroup corresponding to $\Delta$.  
Conjugation by $\Delta$ induces an isomorphism of group schemes $\mathcal{G} \cong \mathcal{G}'$, where $\mathcal{G}$ is
the parahoric group scheme corresponding to $\mathcal{P}$, and $\mathcal{G}'$ is the group scheme obtained by invariant
pushforward.  Therefore $L^+\mathcal{P} = L^+\mathcal{G} \cong L^+\mathcal{G}'$, and $L^+\mathcal{G}' \subseteq L^+G(A')$.
\end{proof}
\begin{prop}\label{EXTEN}
For any scheme $S$ and morphism $\psi: N \times S \rightarrow G$ such that $\psi_0= \psi(0, s)$ factors through the Schubert cell $C_w^P$, the morphism 
$\Delta \psi: N^* \times S \rightarrow G/P$ can be uniquely extended to $N \times S$.
\end{prop}
\begin{proof}  
By assumption $\psi_0$ factors through $BwP \subseteq G$.  Now as shown in section 8.3 of 
\cite{SLAG}, $U_{w^{-1}} \times P \cong BwP$, where $U_{w^{-1}}$ is a subgroup of the unipotent radical of $B$, and the 
isomorphism is given by $(u,p) \mapsto uwp$.  Let $f_0$ and $g_0$ be the compositions of $\psi_0$ and the 
projections to $P$ and $U_{w^{-1}}$, and let $\psi' = w f_0$, extended to
$N \times S$.  Clearly then, since $\Delta$ is a one-parameter subgroup of $T$, $(\psi')^{-1} \Delta^{-1} \psi'$ maps 
to $P$.   Therefore $\Delta \psi \cdot (\psi')^{-1} \Delta^{-1} \psi'$ composed with 
$G \rightarrow G/P$ is equal to $\Delta \psi$.  But $\psi \cdot (\psi')^{-1}$ is
just $g_0: S \rightarrow B$ at $\omega = 0$, and therefore by Lemma \ref{NOPOLE} $\Delta \psi \cdot (\psi')^{-1} \Delta^{-1} \psi'$
is defined for all of $N \times S$.  Since $G/P$ is projective, the extension is clearly unique.
\end{proof}
\begin{cor} \label{PREDMORPH}
There exists a morphism of stacks $\mathcal{C} \rightarrow \normtext{Bun}_Y^{\Gamma, G;P}(\boldsymbol{\tau})$.
\end{cor}
\begin{proof}
By the above proposition, we have constructed a $P$-reduction of $E$ over $Y \times \widetilde{S}$.  The descent data
for $\mathcal{E}$ gives descent data for $E$, and it is easy to see that the $P$-reduction we constructed descends to 
$E$ over $Y \times S$.  Let $\mathcal{T}$ be the stack over $\mathcal{C}$ adding the data of a trivialization near each 
ramification point making the flag trivial.  Then we have constructed a morphism 
$\mathcal{T} \rightarrow \normtext{Bun}_Y^{\Gamma, G;P}(\boldsymbol{\tau})$.
It is well known that $\mathcal{T} \rightarrow \mathcal{C}$ is a torsor with respect to the action on $\mathcal{T}$ by
\[
\prod^n L^+\mathcal{I},
\]
and therefore it suffices to show our construction does not depend on the choice of trivialization (see for example \cite{LASOR}).
Using the same notation as above, a change in trivialization of $\mathcal{E}$ multiplies $\psi$ on the left by some $f \in L^+\mathcal{I}(\widetilde{S})$.
But then $\Delta f \psi = \Delta f \Delta^{-1} \Delta \psi$ and $\Delta f \Delta^{-1}$ corresponds to a change of 
trivialization of $E$, since $\Delta \mathcal{I} \Delta^{-1} \subseteq G(A')^\Gamma$.  Therefore the $P$-reduction
does not depend on the choice of trivialization, finishing the proof.
\end{proof}

We also want to identify the degree and local type of the reductions in the image of this morphism.

\begin{prop}
The morphism $\mathcal{C} \rightarrow \normtext{Bun}_Y^{\Gamma, G;P}(\boldsymbol{\tau})$
factors through $\equivBunPTwo{0}$, the substack of degree $0$ $P$-reductions of local type $\boldsymbol{\tau}_u$.
\end{prop}
\begin{proof}
Let $\widetilde{\mathcal{E}}$ be a parabolic bundle and $\mathcal{F}$ be
a $P$-reduction, together corresponding to a point in $\mathcal{C}$.  Let $E$ be the corresponding 
$(\Gamma, G)$-bundle, and $F$ be the corresponding $\Gamma$-invariant $P$-reduction of $E$.
Then Teleman and Woodward showed that the degree of $F$ is a positive scalar multiple of the parabolic degree (see section \ref{Moduli}) of the
original $P$-reduction $\mathcal{F}$ \cite{TWPB}.  By the assumption that $\vec{w}$ is
on the face of the multiplicative polytope corresponding to $\sigma_{w_1} \ast \cdots \ast \sigma_{w_n} = q^d \text{[pt]}$
 the parabolic degree is $0$, so the degree of $F$ is also $0$.

Following the proof of Proposition \ref{EXTEN}, if the $P$-reduction $\mathcal{F}$ is given near $x$ by $\psi: U_x \rightarrow G/P$,
then the $P$ reduction of $E$ is given by $\Delta_y \psi$.  That is, locally the $P$-reduction corresponds to 
a map:
\begin{align*}
(F_y)_{|N_y^*} &\xrightarrow{} (E_y)_{|N_y^*} \\
(\omega, p) &\mapsto  (\omega, \Delta_y(\omega) \psi(z) p).
\end{align*}
Now say $\psi(0) = bu_ip$.  Then the completion of $\Delta_y \psi$ is $\Delta_y \psi p^{-1} u_i^{-1} \Delta_y^{-1} u_i p$.
Note that $f(\omega) = \Delta_y \psi p^{-1} u_i^{-1} \Delta_y^{-1}$ is in $G(A')^{\Gamma_y}$, since $\psi p^{-1} u_i^{-1}$ is
in the Iwahori subgroup $\mathcal{I}$, and therefore $f(0)$ is in the centralizer $C_G(\tau_i)$.  

Now the $P$-reduction given by $\phi = fu_ip$ is $\Gamma_y$-invariant, which means that for every $\gamma \in \Gamma_y$ 
there is a $p(\gamma, \omega) \in P(A')$ such that 
\[
\tau_i(\gamma) \phi(\omega) = \phi(\gamma \omega) p(\gamma, \omega).
\]
Then the induced $\Gamma_y$ action on $F_y$ is given by $p(\gamma, \omega) = \phi(\gamma \omega)^{-1}\tau_i(\gamma) \phi(\omega)$, and
therefore changing trivializations of $F_y$ as in \cite{TWPB}, we see that the local type of $F$ is 
$\phi(0)^{-1}\tau_i(\gamma) \phi(0) = p^{-1} u_i^{-1} f(0)^{-1} \tau_i f(0) u_i p = p^{-1} u_i^{-1}  \tau_i u_i p$,
finishing the proof.
\end{proof}

\subsubsection{Proof of Theorem \ref{PLIFT}}

We've proven that we have a morphism of stacks $\mathcal{C} \rightarrow \equivBunPTwo{0}$.  It remains to show this
morphism is representable and surjective.

\begin{proof}[Proof of Theorem \ref{PLIFT}]
Consider the following diagram, where $\mathcal{Y}_0$ is the pullback of $\equivBunPTwo{0}$.
\[
\includegraphics{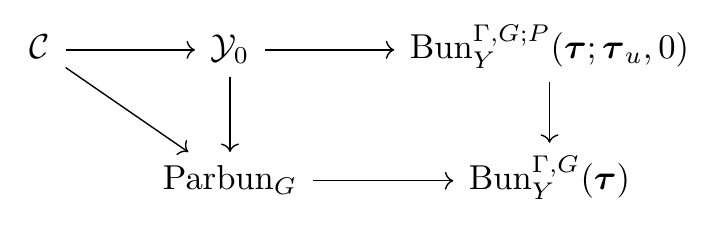}
\]
Then since $\text{Parbun}_G \rightarrow \equivBun$ is representable, $\mathcal{Y}_0 \rightarrow \equivBunPTwo{0}$ is
representable.  Therefore we just need to show $\mathcal{C} \rightarrow \mathcal{Y}_0$ is representable.  We claim
that this morphism is a monomorphism and locally of finite type, and therefore representable (by schemes, see \cite[{Cor 8.1.3}]{LMB} \cite[{Tag 0B89}]{STKS}).  The morphism
is locally of finite type since $\mathcal{C} \rightarrow \text{Parbun}_G$ is locally of finite type.  Suppose we have two parabolic bundles 
$\widetilde{\mathcal{E}}_1$, $\widetilde{\mathcal{E}}_2$, with $P$-reductions $\phi_i: X \times S \rightarrow \mathcal{E}_i$
giving morphisms $f_1,f_2:S \rightarrow \mathcal{C}$.  Let $E_1$, $E_2$ be the corresponding equivariant bundles, and $\phi_i^Y$
the corresponding invariant $P$-reductions.  Then to show that $\mathcal{C} \rightarrow \mathcal{Y}_0$ is a monomorphism,
it is sufficient (by definition) to show that an isomorphism $\widetilde{\mathcal{E}}_1 \cong \widetilde{\mathcal{E}}_2$
identifies $\phi_1$ and $\phi_2$ if and only if it identifies $\phi_1^Y$ and $\phi_2^Y$.  But since $E_i$ is just the pullback
of $\mathcal{E}_i$ away from $p_1, \ldots, p_n$, clearly this is true generically.  Then since $\mathcal{E}_i/P$ and $E_i/P$
are separated over $X \times S$, if the $P$-reductions are identified away from $p_1, \ldots, p_n$, they are the same over
all of $X \times S$.  Therefore $\mathcal{C} \rightarrow \mathcal{Y}_0$ is a monomorphism and representable.

\emph{Second proof:} We sketch a second proof that provides a local description of $\mathcal{C}$ in $\mathcal{Y}_0$,
and additionally shows that $\mathcal{C}$ is immersed in $\mathcal{Y}_0$.
 Suppose we have a morphism $S \rightarrow \mathcal{Y}_0$.  
This morphism corresponds to the following data: a family of parabolic bundles $\widetilde{\mathcal{E}}$
over $S$, a corresponding $\Gamma$-equivariant family of $G$-bundles
$E \rightarrow Y \times S$ and a $\Gamma$-invariant $P$-reduction $Y \times S \rightarrow E/P$, with the given local types.
Then passing to an \'{e}tale cover of $S$ we can trivialize $\mathcal{E}$ over $X^* = X \setminus \{p_1, \ldots, p_n\}$
and formal neighborhoods $x \in U_x$ for each branch point $x \in \{p_1, \ldots, p_n \}$ so that the flags are trivial.
This induces a trivialization of $E$ over $Y^*$ and formal neighborhoods $N_y$ for each ramification point $y$.
  Say the $P$-reduction near $y$ is given
by $\phi: N_y \times S \rightarrow G$, so that letting $F$ be the $P$-bundle corresponding to the $P$-reduction
of $E$ the $P$-reduction then corresponds to
\begin{align*}
(F)_{|N_y} &\rightarrow (E)_{|N_y} \\
(\omega, s, p) &\mapsto  (\omega, s, \phi(\omega, s)p),
\end{align*}
where the $\Gamma$-action on $E$ and $F$ are constant with respect to $\omega$ and $s$ and given by 
$\tau$ and $\tau_u$, respectively, passing to another \'{e}tale cover if necessary.

Now it is easy to see that $\phi(\omega,s) u_i^{-1}$ gives a morphism $S \rightarrow L^+G(A')^\Gamma$.
Let $\Delta_y$ be the rational OPS as above, and let $B'$ be the image of $\Delta_y \mathcal{I} \Delta_y^{-1}$ in $G$.  Then $B'$
is contained in a Borel subgroup $B_\mu$ of $G$, where $B_\mu = w_\mu B w_\mu$ (see Proposition \ref{CONNGRP}).  So letting $P_\mu = w_\mu P w_\mu$
we see that $\phi(\omega,s) w_\mu$ gives a well-defined morphism to $C_G(\tau_i)u_i w_\mu P_\mu/P_\mu$, and that $B'u_i  w_\mu P_\mu/P_\mu$
is contained in $C_G(\tau_i)u_i w_\mu P_\mu/P_\mu$.  Then the pullback $S_\mathcal{C}$ of $S \rightarrow C_G(\tau_i)u_i w_\mu P_\mu/P_\mu$ 
to $B'u_i  w_\mu P_\mu/P_\mu$ for each $i$ is the scheme representing the stack-theoretic fiber
product of $S \rightarrow \mathcal{Y}_0$ and $\mathcal{C} \rightarrow \mathcal{Y}_0$.

\emph{Surjectivity:} Now to show the morphism is surjective, suppose we have an equivariant bundle $E$ over 
$Y \times \text{Spec}(k')$ with an invariant $P$-reduction.  Then since $G(A')^\Gamma$ surjects onto $C_G(\tau)$, we can choose 
trivializations of $E$ such that the $P$-reduction gives elements $x \in B'u_i  w_\mu P_\mu/P_\mu$
over the ramification points of $Y$.  Then taking the transition functions of $E$ with respect to this trivialization, and modifying
them by $\Delta_y^{-1}$, we get transition functions for a $G$-bundle over $X$.  Taking the trivial flags we get a parabolic bundle 
$\widetilde{\mathcal{E}}$ which maps to $E$.  By the above work the $P$-reduction of $E$ corresponds a $P$-reduction of $\mathcal{E}$ 
giving a point in $\mathcal{C}$.
\end{proof}

\subsection{Properness over the semistable locus}

Now we can set up the properness calculation.  
Let $\equivBunPTwo{0}$ be as above.
The letters $ss$ will mean we are working with semistable objects with respect to the given weight data $\vec{w}$; when
it appears on $\text{Parbun}_G$ we mean the inverse image of the semistable locus of $\text{Bun}_{\mathcal{G}}$.
Then by the above work there exists an embedding 
$\mathcal{C} \hookrightarrow \mathcal{Y}_0 = \text{Parbun}_G  \times_{\equivBun}  \equivBunPTwo{0}$ 
making the following diagram commute:
\[
\includegraphics{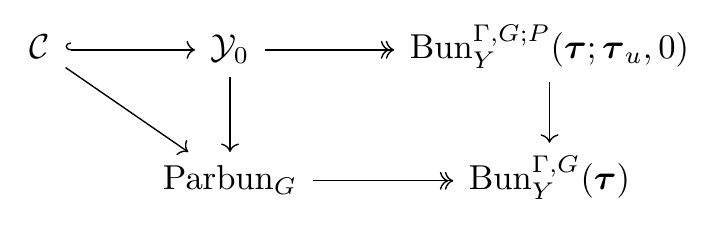}.
\]
Let $\mathcal{Y}$ be the closure of $\mathcal{C}$ in $\mathcal{Y}_0$, which is a valid construction since
$\mathcal{C} \rightarrow \mathcal{Y}_0$ is representable.
Like $\mathcal{C}$, $\mathcal{Y}$ is an integral algebraic stack over $\text{Parbun}_G$.  
Then from this diagram it is clear that if the projection
$\equivBunPTwoSS{0} \rightarrow \equivBunSS$  is proper,
then $\mathcal{Y}^{ss} \rightarrow \text{Parbun}_G^{ss}$ is proper.

Let $\mathcal{R}_0$ be the stack of all degree $0$ $P$-reductions over $\equivBun$.  The following proposition
is a standard application of the theory of Quot schemes.

\begin{prop}\label{PREDREP}
The morphism $\mathcal{R}_0 \rightarrow \equivBun$ is representable, separated and of finite type, and
the stack $\equivBunPTwo{0}$ is a closed substack of the stack $\mathcal{R}_0$.
\end{prop}
\begin{proof}
Let $S$ be noetherian scheme, and consider the fiber $\mathcal{W}$
of $\mathcal{R}_0 \rightarrow \equivBun$  over a morphism $S \rightarrow  \equivBun$.  Then $\mathcal{W}$ is
isomorphic to the stack $\mathfrak{Mor}_{Y \times S}(Y \times S, \mathcal{E}/P)$ 
of sections of $\mathcal{E}/P \rightarrow Y \times S$.  But since $Y \times S$ and $\mathcal{E}/P$ are strongly projective
(in the sense of Altman and Kleiman) over $Y \times S$, $\mathfrak{Mor}_{Y \times S}(Y \times S, \mathcal{E}/P)$ is representable
by a quasi-projective scheme over $Y \times S$ (see \cite{AK80}).
Therefore $\mathcal{R}_0 \rightarrow \equivBun$ is representable, separated and of finite type.

For the second statement: the substack of $\Gamma$-invariant $P$-reductions is closed because 
$\mathfrak{Mor}_{Y \times S}(Y \times S, \mathcal{E}/P)$ is separated, and the substack of invariant
$P$-reductions with fixed local type $\boldsymbol{\tau}_u$ is closed because the conjugacy class of 
each $u_i^{-1} \tau_i u_i$ is closed in $P$, since $u_i^{-1} \tau_i u_i$ is semisimple.
\end{proof}
 
Therefore $\equivBunPTwoSS{0} \rightarrow \equivBunSS$ is proper if $\mathcal{R}_0^{ss} \rightarrow \equivBunSS$ is proper, and
it is sufficient to show that $\mathcal{R}_0^{ss} \rightarrow \equivBunSS$ satisfies the existence part of
the valuative criterion for properness to complete the proof of Theorem \ref{PropThm}.

The following proposition is an easy consequence of the main lemma in \cite{HNA}, which is used to prove a no-ghosts
theorem similar to the one we need, the main difference being that our $G$-bundle is not fixed.

\begin{prop}
Suppose $C$ is a DVR with an algebraically closed residue field, and that $X$ and $Y$ are integral schemes, flat and projective
of relative dimension 1 over $C$, with $Y$ furthermore smooth over $C$.  Suppose we
have a flat $C$-morphism $f:X \rightarrow Y$ that is an isomorphism over $C^*$, and $\xi$ is a relatively ample 
line bundle on $X$.  

Then if the restriction $f_0: X_0 \rightarrow Y_0$ of $f$ to the closed point of $C$ is not an
isomorphism, there is a unique component $D$ of $X_0$ such that $f_0: D_{\text{red}} \rightarrow Y_0$ is an isomorphism
and $\text{deg}(D_{\text{red}}, \xi) < \text{deg}(X_{|C^*}, \xi)$.
\end{prop}

Now we are ready to complete the properness proof.

\begin{proof}[Proof of Theorem \ref{PropThm}]
Now suppose we have the following diagram, where $C$ is the spectrum of a complete discrete valuation ring with an
algebraically closed residue field $k'$, and $C^*$ is the spectrum of its quotient field.
\[
\includegraphics{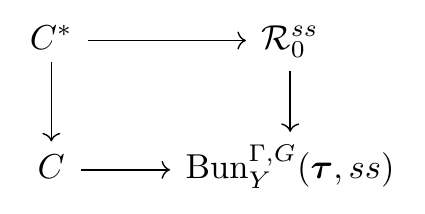}
\]
In order to prove the right vertical arrow is universally closed, it is sufficient (see \cite{EGA2} II.7.3.8 and \cite{LMB} Theorem 7.3 and 7.5)
to show that we can find a lift $C \rightarrow  \mathcal{R}_0^{ss}$ making the above diagram
2-commutative.
This diagram corresponds to a family of semistable $(\Gamma, G)$-bundles $\mathcal{E} \rightarrow Y \times C$ and a family
of degree $0$ $P$-reductions $\phi :Y \times C^* \rightarrow \mathcal{E}/P$.  Now since $\mathcal{E}/P$ is
projective over $C$, we can complete the subscheme $\phi(Y \times C^*)$ to a closed subscheme $Z \subseteq \mathcal{E}/P$,
with $Z$ flat over $C$.  Our goal is to show that this subscheme corresponds to a section of $\mathcal{E}/P \rightarrow Y \times C$.

We claim that $f:Z_0 \rightarrow Y_0 \cong Y \times \text{Spec}(k')$ is an isomorphism.
Suppose not.  Then by the above proposition there is a unique component $D$ of $Z_0$ such that $f: D_{\text{red}} \rightarrow Y_0$
is an isomorphism.
Now let $T_\pi$ be the tangent bundle along the fibers of $\mathcal{E}/P \rightarrow Y \times C$, and
let $\xi$ be the restriction of the determinant of this bundle to $Z_0$.  Then $\xi$ is ample and therefore by the above
proposition we have $\text{deg}(D_{red}, \xi) < \text{deg}(Z_{|C^*}, \xi)$.  
But by assumption $\text{deg}(Z_{|C^*}, \xi)=0$ and therefore $\text{deg}(D_{red}, \xi) < 0$, which violates the 
$\Gamma$-semistability of $\mathcal{E}_{|p_0}$.  
(It is well-known that a bundle is $\Gamma$-semistable if and only if it is semistable; see the proof of Proposition \ref{SEMIEQ}.)   
Therefore $f:Z_0 \rightarrow Y_0$ is an isomorphism.

Finally we need to prove that the map $f: Z \rightarrow Y \times C$ is an isomorphism.  Now we know that 
$Z$ is integral, since it is the closure of a subscheme of $\mathcal{E}/P$ isomorphic to $Y \times C^*$.  Furthermore
$f$ is birational by assumption.  Let $V$ be the open subset of points $x$ in $Y \times C$ such that their
fibers $f^{-1}(x)$ are zero-dimensional.  Then the restriction $f: f^{-1}(V) \rightarrow V$ is projective
and quasi-finite, and therefore finite.  Then since it is also birational and $V$ is normal, it is an isomorphism
by Zariski's main theorem: see Lemme 8.12.10.1 in \cite{EGAIV3}. Therefore $V$ is contained in the largest open set $U$ such that there exists a morphism $U \rightarrow Z$
representing the birational inverse of $f$.  But then by the above work, $Y_0 \subseteq V \subseteq U$.
Therefore $f: Z \rightarrow Y \times C$ is an isomorphism, proving that we have a lift $C \rightarrow  \mathcal{R}_0^{ss}$,
finishing the proof of Theorem \ref{PropThm}.
\end{proof}

\section{Conclusion of the proof of the reduction theorem}\label{ProofSec}

In the previous section we constructed the reduction stack $\mathcal{Y}$, and showed that
restricted to semistable bundles the projection morphism $\mathcal{Y} \rightarrow \text{Parbun}_G$ is proper.
We now use $\mathcal{Y}$ to prove the reduction theorem (Theorem \ref{MNTHM2}).
We continue to assume that $X \cong \mathbb{P}^1$, and $\vec{w}$ is weight data in the multiplicative polytope lying on the
face of the multiplicative polytope corresponding to $\sigma_{w_1} \ast \cdots \ast \sigma_{w_n} = q^d\text{[pt]}$. 

We prove the reduction theorem in two steps: first, we reduce to the Levi subgroup $L$ of $P$ using the properties
of $\mathcal{Y}$ proven in section \ref{PReductionSec}.  Then we reduce to the derived subgroup $L' = [L,L]$ using
an argument similar to the one in section 7 of \cite{BK}.  For an outline of the strategy of the proof,
see the discussion at the beginning of section \ref{PReductionSec}.

\subsection{Reduction to the Levi subgroup $L \subseteq P$}

The first step is to lift conformal blocks to $\mathcal{Y}$.

\begin{prop}
\label{ZarIsom2}
Suppose $f: \mathcal{X} \rightarrow \mathcal{Y}$ is a representable morphism of Artin stacks, where $\mathcal{Y}$ is smooth
over $k$, 
$\mathcal{X}$ is integral, and $f$ is birational and proper.  (By birational we simply mean that there is a non-empty
open substack $\mathcal{U} \subseteq \mathcal{X}$ such that $f$ restricted to $\mathcal{U}$ is an isomorphism onto its image.)
Then for any line bundle $\mathcal{L}$ over $\mathcal{Y}$, the pullback via $f$ induces an isomorphism of global
sections: $\normtext{H}^0(\mathcal{Y}, \mathcal{L}) \xrightarrow{\sim} \normtext{H}^0(\mathcal{X}, \mathcal{L})$.
\end{prop}
\begin{proof}
Let $U \rightarrow \mathcal{Y}$ be a smooth morphism from a connected (and therefore irreducible) scheme $U$.
First we show that, using Zariski's main theorem, the pullback of global sections is an isomorphism 
$\text{H}^0(U, \mathcal{L}) \xrightarrow{\sim} \text{H}^0(V, \mathcal{L})$,
where $V$ is the pullback of $U$ to $\mathcal{X}$.

Assume that $f$ is birational and proper, and that $\mathcal{X}$ is integral.  Let $V \rightarrow \mathcal{X}$ 
be the pullback of $U$.  Our first goal is to show that $V$ is irreducible.  Let $X = |\mathcal{X}|$
and $Y = |\mathcal{Y}|$ be the sets of points of these stacks with the Zariski topology. Now by definition $|V| \rightarrow X$ 
and $|U| \rightarrow Y$ are continuous, open, and surjective maps.  Furthermore by assumption, the map $X \rightarrow Y$ is 
an isomorphism over some open set $T \subseteq Y$, and therefore $|V| \rightarrow |U|$ is an isomorphism
over some $S \subseteq |U|$, since $|V|$ is the fiber product of $X$ and $|U|$ over $Y$.  Now suppose we have two non-empty 
open sets $V_1, V_2 \subseteq |V|$.  Then their images in $X$ must intersect in $T$, since $X$ is irreducible.  But then
$V_1$ and $V_2$ must both intersect $S$, which is irreducible since $|U|$ is irreducible.  Therefore they must intersect, so $V$ is
irreducible.

So $V \rightarrow U$ is a birational and proper morphism of integral schemes over $k$, with $U$ smooth over $k$.
Then by the projection formula, $f_*f^* \mathcal{L} \cong \mathcal{L} \otimes f_*\mathcal{O}_V$.  Now since $f$ is proper,
$f_*\mathcal{O}_V$ is a coherent $\mathcal{O}_U$-module, and in fact we have $\mathcal{O}_U \subseteq f_*\mathcal{O}_V \subseteq K$
where $K$ is the function field of $U$ and $V$.  But since $U$ is nonsingular, it is in particular normal, and so the structure
sheaf $\mathcal{O}_U$ is locally integrally closed, and since $f_*\mathcal{O}_V$ is coherent, we have $f_*\mathcal{O}_V = \mathcal{O}_U$.
But then $\text{H}^0(V, f^*\mathcal{L}) = \text{H}^0(U, f_*f^*\mathcal{L}) = \text{H}^0(U, \mathcal{L})$.

Now let a collection of smooth morphisms $U_i \rightarrow \mathcal{Y}$ as above be jointly surjective,
and let $\mathbf{U} = \bigsqcup_i U_i$.  Then since both $\mathcal{X}$ and $\mathcal{Y}$ are reduced,
it is easy to see that the isomorphism $\text{H}^0(\mathbf{U}, \mathcal{L}) \xrightarrow{\sim} \text{H}^0(\mathbf{V}, f^*\mathcal{L})$
descends to $\text{H}^0(\mathcal{Y}, \mathcal{L}) \xrightarrow{\sim} \text{H}^0(\mathcal{X}, \mathcal{L}).$
\end{proof}

\begin{prop} \label{YISOM}
We have via the natural pullback map, an isomorphism: 
\[
\normtext{H}^0(\normtext{Parbun}_G, \mathcal{L}_{\vec{w}}) \xrightarrow{\sim} \normtext{H}^0(\mathcal{Y}, \mathcal{L}_{\vec{w}}).
\]
\end{prop}
\begin{proof}
It was shown in \cite{BK} that $\mathcal{C} \rightarrow \text{Parbun}_G$ is birational.  Since by Theorem \ref{PropThm} the morphism 
$\mathcal{Y}^{ss} \rightarrow \text{Parbun}_G^{ss}$ is proper, by Proposition \ref{ZarIsom2} we have an isomorphism
$\normtext{H}^0(\text{Parbun}_G^{ss}, \mathcal{L}_{\vec{w}}) \xrightarrow{\sim} \normtext{H}^0(\mathcal{Y}^{ss}, \mathcal{L}_{\vec{w}})$.
By Theorem \ref{GlblSecDesc}, Proposition \ref{SEMIEQ}, and \cite[Theorem 9.10]{QUANT},  the left vertical arrow in the following diagram is an isomorphism:
\[
\includegraphics{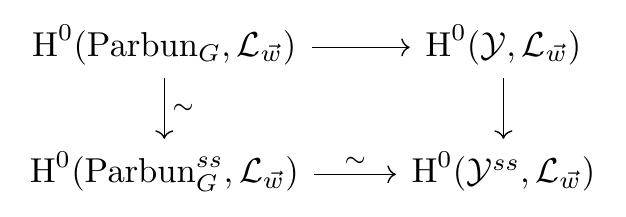}.
\]
By commutativity of this diagram, the right vertical arrow is surjective.  Furthermore, the right vertical arrow is injective, 
since $\mathcal{Y}$ is integral and $\mathcal{Y}^{ss}$ is a nonempty open substack.
Note we're assuming $\vec{w}$ is in the eigen-polytope, so that some power of $\mathcal{L}_{\vec{w}}$ has global sections.  
Then the right vertical arrow is an isomorphism, and therefore we have
$\text{H}^0(\text{Parbun}_G, \mathcal{L}_{\vec{w}}) \xrightarrow{\sim} \text{H}^0(\mathcal{Y}, \mathcal{L}_{\vec{w}}).$

\end{proof}

Now let $P$ be the maximal parabolic associated to the product $\sigma_{u_1} \ast \cdots \ast \sigma_{u_n} = q^d[pt]$, 
and $L \subseteq P$ the Levi factor containing the maximal torus of $G$.  Let $\equivBunL{d}$ be the moduli stack of $\Gamma$-equivariant
$L$-bundles of local type $\boldsymbol{\tau}_u=(u_1^{-1}\tau_1u_1, \ldots, u_n^{-1}\tau_nu_n)$ and degree $d$.  Let 
$\xi: \equivBunPTwo{0} \rightarrow \equivBunL{d}$
be the natural projection given by $P \rightarrow L$, and $\iota: \equivBunL{d} \rightarrow \equivBun$ be the morphism given
by extending the structure group to $G$.

Then we have the following proposition.

\begin{prop} \label{PULLBCK}
The following diagram
\[
\includegraphics{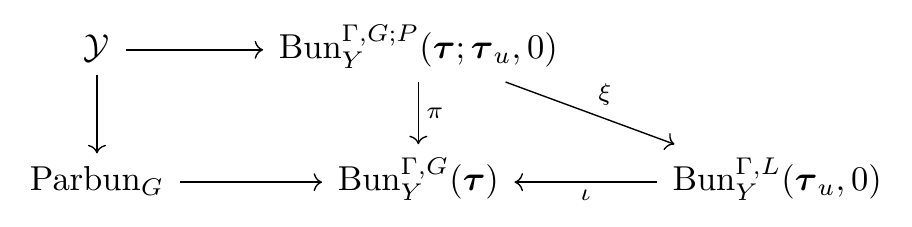}
\]
induces a commutative diagram of global sections of $\mathcal{L}_{\vec{w}}$.  Furthermore $\xi$ is surjective.
\end{prop}
\begin{proof}
The morphism $\xi$ is surjective because there is a section $j: \equivBunL{0} \rightarrow \equivBunPTwo{0}$ given
by extension of structure group.

The left square is 2-commutative, so we just need to show the triangle on the right induces a commutative diagram
of global sections.  We use the methods of \cite{BK}, translated to equivariant bundles.  Let $N_P$ be the smallest positive integer such 
that $N_P x_P$ is in the coroot lattice, and let $\bar{x}_P = N_P x_P$.  Let $t^{\bar{x}_P}$ be the associated 
cocharacter, and $\phi_t: P \rightarrow P$ be the homomorphism sending $p \mapsto t^{\bar{x}_P} p t^{-\bar{x}_P}$.
Then the family of homomorphisms $\phi: P \times \mathbb{G}_m \rightarrow P$ extends to a family over $\mathbb{A}^1$,
and $\phi_0:P \rightarrow P$ factors through $L \subseteq P$.

Now assume we have a morphism $S \rightarrow \equivBunPTwo{0}$, corresponding to an equivariant $G$-bundle $E$ and 
a $P$-reduction $\sigma: Y \times S \rightarrow E/P$, which in turn corresponds to an equivariant $P$-bundle $F$ with local type
$\boldsymbol{\tau}_u$.  Then let $F_t = F \times^{\phi_t} P$, and $E_t = F_t \times^P G$.  Clearly then $(E_1, F_1)$
is isomorphic to the pair $(E,F)$ (in fact the same is true for $(E_t, F_t)$ for any $t \neq 0$), and $E_0$
has a reduction to an $L$-bundle, which we will also denote $F_0$.  This is what Belkale and Kumar call the 
\emph{Levification process}.

So we get a canonical morphism $f: S \times \mathbb{A}^1 \rightarrow \equivBunPTwo{0}$, and $F_0$ is (isomorphic to) the
image of $S \rightarrow \equivBunPTwo{0}$ via $\xi$.  Let $\mathcal{L}_1$ be the pullback of $\mathcal{L}_{\vec{w}}$ via $\pi$,
and $\mathcal{L}_2$ the pullback via $\iota \circ \xi$.  Clearly then there is a canonical isomorphism
 $f_1^* \mathcal{L}_2 \cong f_0^* \mathcal{L}_1$.  Our goal is to show that this is a canonical isomorphism 
$f_1^* \mathcal{L}_1 \cong f_0^* \mathcal{L}_1$, which will complete the proposition.  We will do this using
a $\mathbb{G}_m$-action on $f^*\mathcal{L}_1$, which we will show is trivial over $t=0$.

The equivariant $\mathbb{G}_m$ action is defined as follows.  There is a natural $\mathbb{G}_m$-action on 
the $G$-bundle corresponding to $f$, defined over $\mathbb{A}^1$ by right multiplication by $t^{\bar{x}_P}$. 
This induces a $\mathbb{G}_m$-action on $f^*\mathcal{L}_1$.  Now Belkale and Kumar show in section 6 of \cite{BK} that
 there is some $N>0$ such that $\mathcal{L}_{\vec{w}}^N$
is isomorphic to the determinant bundle $D(V)$, where $V$ is the adjoint representation of $G$.
Then $f_0^* \mathcal{L}_1^N \cong D(F_0 \times^L \mathfrak{g})$.

Let $s \in S$ be a point and $(F_0)_s$ be the fiber over $s$.  Now we can filter $\mathfrak{g}$ so that the action
of $\text{Ad}(t^{\bar{x}_P})$ on the associated graded pieces is $t^{-\gamma}$, for some integer $\gamma$.  Denote
the associated graded piece with an action of $t^{-\gamma}$ as $\mathfrak{g}_\gamma$, and $(F_0)_s \times^L \mathfrak{g}_\gamma$
as $V_\gamma$ (note that $\mathfrak{g}_\gamma$ is fixed by $L$, since $t^{\bar{x}_P}$ is central).  Then $\mathbb{G}_m$ acts by 
$t^{-\gamma}$ on $V_\gamma$, and therefore by $t^{\chi(Y, V_\gamma) \gamma}$ on $D(V_\gamma)$.  Then the exponent of the action on 
$D((F_0)_s \times^L \mathfrak{g})$  is
\begin{align}
\sum_\gamma \chi(Y, V_\gamma) \gamma &= \sum_\gamma (\text{deg}(V_\gamma) + (1-g) \text{rk}(V_\gamma)) \gamma \\
                                     &= \sum_\gamma \text{deg}(V_\gamma) \gamma
\end{align}
where the first equality is Riemann-Roch, and the second follows from the fact that 
$\text{dim } \mathfrak{g}_\gamma = \text{dim } \mathfrak{g}_{-\gamma}$.

Now let $R_\gamma \subseteq R$ be the set of roots of $\mathfrak{g}_\gamma$.  Then clearly
\[
\text{deg}(V_\gamma) = \sum_{\alpha \in R_\gamma} \text{deg}((F_0)_s \times^L \mathbb{C}_{-\omega_P}) \cdot \alpha(\alpha^\vee_P),
\]
noting that $\alpha(\alpha^\vee_P)$ is the coefficient of $\omega_P$ in $\alpha$.  But $d = \text{deg}((F_0)_s \times^L \mathbb{C}_{-\omega_P})$
is zero.  Therefore the $\mathbb{G}_m$-action on $D(F_0 \times^L \mathfrak{g})$ 
is trivial, which implies the action on $f_0^* \mathcal{L}_1$ is trivial.  Then the $\mathbb{G}_m$-action gives a canonical identification 
$f_1^* \mathcal{L}_1 \cong f_0^* \mathcal{L}_1$ by taking the limit $t \rightarrow 0$.
\end{proof}

Now we are ready to reduce to the Levi.

\begin{cor}
\label{LRed}
We have $\normtext{H}^0(\equivBunL{0}, \iota^*\mathcal{L}_{\vec{w}}) \xleftarrow{\sim} \normtext{H}^0(\normtext{Bun}_\mathcal{G}, \mathcal{L}_{\vec{w}}) \cong \normtext{H}^0(\normtext{Parbun}_G, \mathcal{L}_{\vec{w}})$.
\end{cor}
\begin{proof}
Since $\xi$ and $\mathcal{Y} \rightarrow \equivBunPTwo{0}$ are surjective by the above proposition and Theorem \ref{PLIFT}, 
$\text{H}^0(\equivBunL{0}, \mathcal{L}_{\vec{w}}) \rightarrow \text{H}^0(\mathcal{Y}, \mathcal{L}_{\vec{w}})$
is injective.  Since by Proposition \ref{YISOM} pullback via $\mathcal{Y} \rightarrow \text{Parbun}_G$ is surjective,
by Proposition \ref{PULLBCK} and the fact that conformal blocks descend to $\equivBun \cong \normtext{Bun}_\mathcal{G}$ 
(Theorem \ref{GlblSecDesc},)  the proof is finished by a simple diagram chase. 
\end{proof}

\subsection{Reductions to $L'$ and completion of the proof of the main theorem}

To complete the proof of the reduction theorem, we need to identify $\normtext{H}^0(\equivBunL{0}, \iota^*\mathcal{L}_{\vec{w}})$
with a space of global sections over $\text{Parbun}_{L'}$.  In order to accomplish this for an arbitrary degree $d$, we
need to add weights to our weight data.  Having done this, we conclude the section with a proof of the identification
of conformal blocks bundles when $d=0$.

We write our weight data for $L$ bundles as $\vec{w}^L=(\lambda_1^L, \ldots, \lambda_n^L, \ell)$, where $\lambda_i^L = u_i^{-1} \lambda_i$.
Note that these weights satisfy the equation $\sum_{i=1}^n \langle \omega_P, \lambda_i^L \rangle = \ell \cdot d$.  

It is easy to see that $L'$ is simply connected, since $G$ is simply connected.  Furthermore we know the Dynkin type of $L'$: it is given
by removing the vertex of the Dynkin diagram of $G$ corresponding to $P$.  Therefore $L' \cong G_1 \times G_2 \times G_3$, where
$G_1$, $G_2$ and $G_3$ are simple, simply connected groups.  Note that one or more of the groups may be trivial, and most commonly
there are exactly two non-trivial factors.  The following discussion follows closely section 7 of \cite{BK}.

Let $Z_0$ be the connected component of the identity of $L$, and let $L' \times Z_0 \rightarrow L$ be the natural
homomorphism.  This homomorphism is in fact an isogeny, with kernel say of size $k_L$.  
Alternatively, $k_L$ is the size of the kernel of the isogeny $Z_0 \rightarrow L/L'$.
Let $N_P$ be the smallest positive integer such that $N_P x_P$ is in the
coroot lattice, and let $\bar{x}_P = N_P x_P$.  Then it is easy to see that 
\[
k_L = \omega_P(\bar{x}_P) = 2 N_P \frac{\langle \omega_P, \omega_P \rangle}{\langle \alpha_P, \alpha_P \rangle}.
\]

The basic result we use to reduce conformal blocks to $L'$ is the following proposition.

\begin{prop} \cite{BK}
Suppose weight data $\vec{w}^L = (\lambda_1^L, \ldots, \lambda_n^L, \ell)$ for $L$ satisfies the equation
$\sum_{i=1}^n \langle \omega_P, \lambda_i^L \rangle = \ell \cdot d$
and that $d = d'k_L$.  Then there is a surjective morphism $\iota': \normtext{Parbun}_{L'} \rightarrow \normtext{Parbun}_L(d)$ 
such that the induced pullback of global sections of $\mathcal{L}_{\vec{w}}$ is an isomorphism.
\end{prop}

To reduce down to $L'$ for general $d$, one needs to change the degree of the $L$ bundles.
For each parabolic $P$ Belkale and Kumar show the existence of an element of the coroot lattice $\mu_P$ lying in the fundamental
alcove of $L$ such that $|\omega_P(\mu_P)| = 1$.   They use $\mu_P$ to shift the degree of the stack of parabolic $L$-bundles,
since for the reduction to $L'$, it is necessary that $k_L$ divides the degree.
Let $d_0$ be the smallest positive integer such that $d+d_0 \omega_p(\mu_P) \equiv 0 \text{ (mod }k_L)$.  
Let $\normtext{Parbun}_{L}^{\lbrack d_0 \rbrack}(d)$ be the stack of parabolic degree $d$ $L$-bundles with full flags over $n+d_0$
points in $X\cong \mathbb{P}^1$.  Let $\mathcal{L}_{\vec{w}^L}$ be the pullback of $\mathcal{L}_{\vec{w}}$ to
$\normtext{Parbun}_{L}^{\lbrack d_0 \rbrack}(d)$ via $\iota$ and the forgetful functor. This is the line bundle associated to a level $\ell$, weights 
$\lambda_1^L, \ldots, \lambda_n^L$, and the zero weight on the remaining $d_0$ points.
Then Corollary $7.6$ in \cite{BK} says the following.

\begin{prop} \cite{BK}
\label{LShiftRed}
Associated to $\mu_P$ is a natural isomorphism 
$\tau_\mu: \normtext{Parbun}_{L}^{\lbrack d_0 \rbrack}(d+d_0 \omega_p(\mu_P)) \rightarrow \normtext{Parbun}^{\lbrack d_0 \rbrack}_L(d)$.
The weights of $\tau_\mu^*\mathcal{L}_{\vec{w}^L}$ are $\lambda_1^L, \ldots, \lambda_n^L$, and
$d_0$ copies of $\ell \cdot \kappa(\mu_P)$, and the level remains the same.
\end{prop}

Note that the forgetful morphism $\normtext{Parbun}^{\lbrack d_0 \rbrack}_L(d) \rightarrow \normtext{Parbun}_L(d)$
induces an isomorphism of global sections for any line bundle for the same basic reason that conformal blocks descend to
stacks of parahoric bundles.
Combining this fact and the above propositions, we can identify global sections of $\mathcal{L}_{\vec{w}^L}$ over $\text{Parbun}_L(d)$
with its pullback to $\text{Parbun}^{[d_0]}_{L'}$ via $\tau_\mu$ and $\iota'$.

There is a morphism $\text{Parbun}_L(d) \rightarrow \equivBunL{0}$, defined in the same way as $\text{Parbun}_G \rightarrow \equivBun$,
so that the pullback of $\iota^*\mathcal{L}_{\vec{w}}$ to $\text{Parbun}_L(d)$ is the line bundle associated to the weight data $\vec{w}^L$.
One way to finish the proof of the reduction theorem would be to show that the pullback of global sections of any line bundle with 
respect to this morphism is an isomorphism.  This could be proven in the same way that we showed that conformal blocks
descend to stacks of parahoric bundles: the geometric fibers of this morphism should be products of quotients of centralizers
in $L$ by Borel subgroups.  Any centralizer of a torus element of $L$ will be reductive and connected, since $L$ is
connected and $L'$ is simply connected.  Unfortunately, we do not have the references in the reductive case to
feel confident in this approach.  

Instead, we simply replicate the above propositions for equivariant bundles.  More precisely, we want to construct 
a morphism $\iota': \equivBunLP \rightarrow \equivBunL{0}$ so that it fits into the following diagram.
\[
\includegraphics{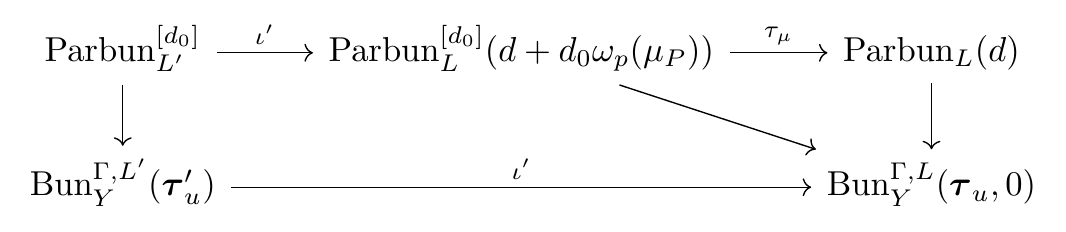}
\]

First let's review the definition of $\iota'$.  Suppose $d + d_0 \omega_p(\mu_P) = d'k_L$, and let $\mathcal{F}$ be an $L'$-bundle.  
Then $F \times \mathcal{O}_X(d')$ is an $L' \times Z_0$ bundle, and therefore extending the structure group via $L' \times Z_0 \rightarrow L$
we get an $L$-bundle $\mathcal{F}_L$ of degree $d + d_0 \omega_p(\mu_P)$.  Parabolic structures are transferred in the
obvious way.  The idea of the construction of $\iota'$ for equivariant bundles is to use an equivariant version of $\mathcal{O}_X(d')$
over $Y$.

There is a canonical identification
of the rational coweights of $L$, and the rational coweights of $L' \times Z_0$.  Therefore, given
a rational coweight $\mu$ of $L$, we can factor it uniquely as $\mu' \cdot \mu''$, where $\mu'$ is
a rational coweight of $L'$, and $\mu''$ is a rational coweight of $Z_0$.  Note that a coweight of $L$
may factor into rational coweight of $L'$ and $Z_0$.

Assume we have chosen $Y$ so that all its ramification indices are divisible by $k_L$, and such that
there are $d_0$ extra ramified orbits of $\Gamma$, with the isotropy subgroup acting trivially over
these points.  This is already necessary for $\equivBunLP$ to be defined.  Let $\mu_1, \ldots, \mu_{n+d_0}$
be the rational coweights associated to $\vec{w}^L$.  In other words, $\mu_i = \frac{1}{\ell}\kappa(\lambda^L_i)$ for
$1 \leq i \leq n$, and $\mu_i = \mu_P$ for $n+1 \leq i \leq n+d_0$.  Then given a parabolic $L'$-bundle
$\mathcal{F}$, the coweights $\mu_1', \ldots, \mu_{n+d_0}'$ allow one to construct the associated equivariant
bundle $F$.  Similarly, the $Z_0$-coweights $\mu_1'', \ldots, \mu_{n+d_0}''$ allow one to construct a (unique)
equivariant line bundle $\mathcal{O}_Y(d', \vec{\mu})$.  Then $\iota': \equivBunLP \rightarrow \equivBunL{0}$
is defined as follows: for any $F \in \equivBunLP$, extend the structure group of $F \times \mathcal{O}_Y(d', \vec{\mu})$
to $L$ via $L' \times Z_0 \rightarrow L$.  It is easy to check that this morphism is well defined and
fits into the above diagram.  Let $\mathcal{L}_{\vec{w}^L}$ be the pullback of $\mathcal{L}_{\vec{w}}$
 via $\iota: \equivBunL{0} \rightarrow \equivBun$; note that this line bundle pulls back to
 $\mathcal{L}_{\vec{w}^L}$ over $\text{Parbun}_L(d)$.
Then we have the following proposition, where we let $\lambda_{n+i} = \ell \kappa(\mu_P)$.

\begin{prop} \label{L'Red}
Suppose weight data $\vec{w}^L = (\lambda_1^L, \ldots, \lambda_{n+d_0}^L, \ell)$ for $L$ satisfies the equation
$\sum_{i=1}^{n+d_0} \langle \omega_P, \lambda_i^L \rangle = \ell \cdot d$. 
Then the morphism $\iota':\equivBunLP \rightarrow \equivBunL{0}$ induces an isomorphism of global sections
of $\mathcal{L}_{\vec{w}^L}$.
\end{prop}
\begin{proof}
Firstly, we note that $\iota'$ is surjective.
It is easy to see that $\text{Parbun}_{L'}^{[d_0]} \rightarrow \equivBunLP$ and 
$\text{Parbun}_{L}^{[d_0]}(d + d_0 \omega_p(\mu_P)) \rightarrow \equivBunL{0}$ are surjective: the first case is well known
since $L'$ is semi-simple.
In the other case, given a $(\Gamma,L)$-bundle $F$ over $Y$, one constructs a parabolic $L$-bundle over $X$ by simply taking
the quotient over $Y^*$, and using \'{e}tale-local trivializations of $F$ over the ramification points to construct 
a parabolic $L$-bundle over $X$, following the above work for $(\Gamma, G)$-bundles.  Note that we do not need a generic
trivialization of $F$ or an understanding of the effect of the choice of trivialization to show the morphism is surjective;
we defer such analysis to future work.  Therefore by the above diagram $\iota'$ is surjective, and therefore the pullback of 
global sections of any line bundle is injective.

To show the pullback of global sections is surjective we follow the proof of Lemma 7.1 in \cite{BK}.  
Assume we have two $(\Gamma, L)$-bundles $F_1$ and $F_2$, and choose lifts to $(\Gamma, L')$-bundles $F_1'$, $F_2'$.
Suppose further we have an isomorphism $\phi: F_1 \xrightarrow{\sim} F_2$.  We want to show we can modify this isomorphism
by multiplication by an element of $Z_0$ such that it lifts to an isomorphism of $F_1'$ and $F_2'$.  This will give a
canonical identification of the fibers of $\mathcal{L}_{\vec{w}}$ and its pullback, since $Z_0$ acts trivially on $\mathcal{L}_{\vec{w}}$
(see proof of Prop \ref{PULLBCK}), and therefore show that the pullback of global sections is surjective. But $\phi$
gives an isomorphism of the associated $L/L'$-bundles, and since $L/L'$ is a torus, the isomorphism therefore corresponds
canonically to some $zL' \in L/L'$.  Some more care could be taken here: the $L/L'$-bundles associated to $F_1'$, $F_2'$ can
be canonically identified with $\mathcal{O}_Y(d', \vec{\mu})$ extended to an $L/L'$-bundle; $\phi$ then induces an 
automorphism of this bundle giving $z$.  But $Z_0 \rightarrow L/L'$ is surjective, 
so we can lift $zL'$ to $z \in Z_0$.  It can be easily checked that composing $\phi$ with the automorphism of $F_2$ 
induced by $z^{-1}$ gives an automorphism that lifts to $\phi': F_1' \rightarrow F_2'$.
\end{proof}

By the results in section \ref{ParahoricSec}, the morphism $\text{Parbun}_{L'} \rightarrow \equivBunLP$ induces an isomorphism
of global sections of $\mathcal{L}_{\vec{w}'}$.  Note that the weights in $\vec{w}'$ are the restrictions of $u_i^{-1} \lambda_i$
to $L'$.  All that remains is the identification of the levels.  

The level(s) of the reduced conformal blocks depends on the Dynkin indices (see the section \ref{CBDEF})  of $L'$ in $G$.  
Let $m_1$, $m_2$, and $m_3$ be the Dynkin indices of each subalgebra $\mathfrak{g}_1$, $\mathfrak{g}_2$, and $\mathfrak{g}_3$ in $\mathfrak{g}$. 
Let $V$ be a faithful representation of $G$, and $D(V)$ be the 
associated determinant bundle over $\text{Parbun}_G$.  The level of $D(V)$ is the Dynkin index of $V$.
Then the pullback of this line bundle to $\text{Parbun}_{G_i}$ is just $D(V_{|G_i})$.  But by the results in section 5 
of \cite{SNR} and section 7 of \cite{BK}, the level of this bundle is the Dynkin index of $V_{|G_i}$, which is equal to 
the index of $V$ times the index of $G_i$ in $G$.  Therefore by linearity pulling back a line bundle $\mathcal{L}$ of
level $\ell$ gives a bundle of level $m_i \ell$ over $\text{Parbun}_{G_i}$.

This completes the proof of Theorem \ref{MNTHM2}.
Finally, we prove that when $d=0$ this isomorphism can be extended to an isomorphism of vector bundles.

\begin{cor}[Corollary \ref{CBBCOR}]\label{CBBPF}
When $d=0$ we in fact have an isomorphism of conformal blocks bundles on $\overline{\text{M}}_{0,n}$:
\[
\mathbb{V}_{\mathfrak{g}, \vec{w}} \cong \mathbb{V}_{\mathfrak{g}_1, \vec{w}_1} \otimes \mathbb{V}_{\mathfrak{g}_2, \vec{w}_2} \otimes \mathbb{V}_{\mathfrak{g}_3, \vec{w}_3}.
\]
\end{cor}
\begin{proof}
Letting $\mathbb{A}_{\mathfrak{g}, \vec{w}}$ be the trivial bundle of invariants over $\overline{\text{M}}_{0,n}$, we
have the following diagram of vector bundles:
\[
\includegraphics{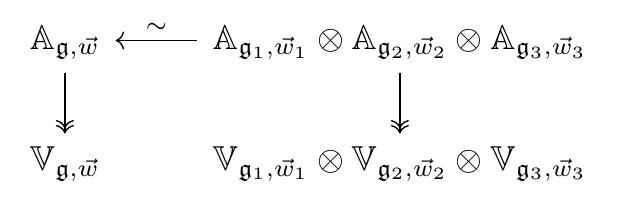}
\]
where the horizontal isomorphism follows from the factorization result for invariants in \cite{RTH}, or alternatively, the
above theorem, choosing a high enough level.  So we just need to show that the composition
$\mathbb{A}_{\mathfrak{g}_1, \vec{w}_1} \otimes \mathbb{A}_{\mathfrak{g}_2,\vec{w}_2} \otimes \mathbb{A}_{\mathfrak{g}_3, \vec{w}_3} \rightarrow \mathbb{V}_{\mathfrak{g}, \vec{w}}$
descends to $\mathbb{V}_{\mathfrak{g}_1, \vec{w}_1} \otimes \mathbb{V}_{\mathfrak{g}_2, \vec{w}_2} \otimes \mathbb{V}_{\mathfrak{g}_3, \vec{w}_3}$, since we've already shown the conformal blocks bundles are
the same rank.  Furthermore, it is sufficient to check this on $\text{M}_{0,n}$, which is dense in $\overline{\text{M}}_{0,n}$, 
and since these are vector bundle morphisms, we can check it fiber by fiber.  The necessary diagram of fibers is induced
by the following diagram:
\[
\includegraphics{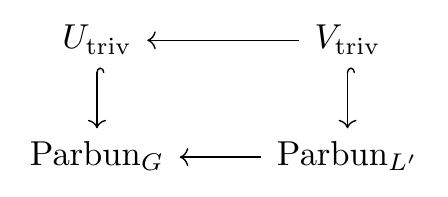}
\]
where $U_{\text{triv}}$ and $V_{\text{triv}}$ are the substacks of trivial bundles, and the diagram of fibers is obtained by taking
global sections of $\mathcal{L}_{\vec{w}}$, then taking the duals of each map.  Note that the fact that $V_{\text{triv}}$
maps into $U_{\text{triv}}$ requires $d=0$.
\end{proof}

The above method fails when $d > 0$, because in this case trivial bundles in $\text{Parbun}_{L'}$ do not map to trivial bundles
in $\text{Parbun}_G$.  We give two examples: the first is on a $d=0$ wall and illustrates the divisor
relation in the above corollary, and the second is on a $d>0$ wall and shows that this corollary
does not directly generalize to these walls.

\begin{examp}
Here we exhibit an example with the degree of the cohomology product equal to zero.  
Let $G = \text{SL}_4$, and $G/P = \text{Gr}(2,4)$.  Let $I_1=I_2=\{2,3\}$ and $I_3 = \cdots = I_6 = \{3,4\}$, 
which correspond to Schubert cells in the usual way.
Then $\sigma_{I_1} \cdots \sigma_{I_6} =  \lbrack \text{pt} \rbrack$, so this product defines a degree $0$
facet of the multiplicative polytope.  Let $\lambda_1 = \lambda_2 = 4\omega_1 + \omega_2 $, and 
$\lambda_3 = \cdots = \lambda_6 = \omega_1 + \omega_3$, and $\ell=5$.  Then this weight data lies 
on the given facet, and the reduced weights are
 $\lambda_1' = \cdots = \lambda_6' = \omega_1$,  $\lambda_1'' = \lambda_2'' =  5\omega_1 $,
and $\lambda_3'' = \cdots = \lambda_6'' = \omega_1$.
Then using Swinarski's conformal blocks package for Macaulay 2 \cite{CBD,M2} one can calculate that 
$\text{rk}(\mathbb{V}_{\mathfrak{sl}_{4}, \vec{\lambda}, 5}) = 10$, 
$\text{rk}(\mathbb{V}_{\mathfrak{sl}_{2}, \vec{\lambda}', 5}) = 5$, and
$\text{rk}(\mathbb{V}_{\mathfrak{sl}_{2}, \vec{\lambda}'', 5}) = 2$.  Five is above the critical level
(see \cite{BGM15}) for $\vec{\lambda}'$, so the corresponding vector bundle is trivial.  
Finally, again using Swinarski's conformal blocks package, one calculates that the symmetrized conformal
blocks divisors of these bundles satisfies the relation in Corollary \ref{CBBCOR}:
$\mathbb{SD}_{\mathfrak{sl}_{4}, \vec{\lambda}, 5} = 5 \cdot \mathbb{SD}_{\mathfrak{sl}_{2}, \vec{\lambda}'', 5} = 1920D_2 + 2160D_3$.
\end{examp}

\begin{examp}\label{POSDEGEX}
We give an example of a reduction on a positive degree wall.  Our group is $\text{SL}_4$, and our Grassmannian $\text{Gr}(1,4)$.
Now it is well-known that $\text{QH}^*(\text{Gr}(1,4)) \cong \mathbb{Z}[T,q]/(T^4 - q)$.  
Let $I_1=\{2\}$ and $I_2 = \cdots = I_6 = \{3\}$ be sets corresponding to Schubert varieties in the usual way,
where $I=\{4\}$ corresponds to the big cell, and $I=\{1\}$ corresponds to a point.
Then $\sigma_{I_1} \ast \cdots \ast \sigma_{I_6} = q \lbrack \text{pt} \rbrack$.  Let 
$\lambda_1 = \cdots \lambda_4 = 2\omega_3$, and $\lambda_5 = \lambda_6 = \omega_1+\omega_3$ and $\ell=2$.
It is easy to see that these weights lie on the wall corresponding to the given product.  Furthermore $k_L = 4$ in this case,
so that $d_0 = 1$. Then $\lambda_1'' = \cdots \lambda_4'' = 2\omega_2$ and 
$\lambda_5'' = \lambda_6'' = \omega_1+\omega_2$, and we add seventh weight $\lambda_7'' = 2 \mu_P^* = 2\omega_1$.
Then using Swinarski's conformal blocks package one calculates that $\text{rk}(\mathcal{V}^{\dagger}_{\mathfrak{sl}_{4}, \vec{\lambda}, 2}) = 1$ and 
$\text{rk}(\mathcal{V}^{\dagger}_{\mathfrak{sl}_{3}, \vec{\lambda}'', 2}) = 1$, satisfying the statement of the theorem.

Increasing the level moves the weight data off the given facet of the polytope, and so we would expect the ranks to
be different in general.  Indeed: $\text{rk}(\mathcal{V}^{\dagger}_{\mathfrak{sl}_{4}, \vec{\lambda}, 3}) = 12$
and $\text{rk}(\mathcal{V}^{\dagger}_{\mathfrak{sl}_{3}, \vec{\lambda}'', 3}) = 24$.  If we raise the level above the critical
level (which is $7$ in both cases) so that the spaces of conformal blocks become isomorphic to spaces of tensor
invariants, we see that $\text{rk}(\mathcal{V}^{\dagger}_{\mathfrak{sl}_{4}, \vec{\lambda}, \ell}) = 21$
and $\text{rk}(\mathcal{V}^{\dagger}_{\mathfrak{sl}_{3}, \vec{\lambda}'', \ell}) = 124$, showing that the two spaces of 
invariants are not isomorphic.

One can compare the divisors over $\overline{\text{M}}_{0,7}$ by adding a zero weight to $\vec{\lambda}$.
Using Swinarski's conformal blocks package one calculates that
$\mathbb{SD}_{\mathfrak{sl}_{4}, \vec{\lambda}, 2} = 1920D_2 + 2880D_3$ and
$\mathbb{SD}_{\mathfrak{sl}_{3}, \vec{\lambda}'', 2} = 3840D_2 + 4320D_3$.
Therefore the isomorphism in the reduction theorem does not extend to an isomorphism of conformal blocks bundles 
in this case.
\end{examp}

\bibliographystyle{abbrv}
\begin{footnotesize}
\bibliography{ConfRed}
\end{footnotesize}

\end{document}